\documentclass[12pt]{article}
\usepackage{amssymb}
\usepackage{amsmath}
\usepackage{latexsym}
\usepackage{amsfonts}
\usepackage{url}
\usepackage{graphicx}
\usepackage{mathrsfs}
\usepackage[usenames]{color}

\newcommand{\bsx}{\boldsymbol{x}}
\newcommand{\bsy}{\boldsymbol{y}}
\newcommand{\bsz}{\boldsymbol{z}}
\newcommand{\bsb}{\boldsymbol{b}}
\newcommand{\bsc}{\boldsymbol{c}}
\newcommand{\bsg}{\boldsymbol{g}}
\newcommand{\bsh}{\boldsymbol{h}}
\newcommand{\bsq}{\boldsymbol{q}}
\newcommand{\bsk}{\boldsymbol{k}}
\newcommand{\bsw}{\boldsymbol{w}}
\newcommand{\bsf}{\boldsymbol{f}}
\newcommand{\bsa}{\boldsymbol{a}}
\newcommand{\bsn}{\boldsymbol{n}}
\newcommand{\bsj}{\boldsymbol{j}}
\newcommand{\bsm}{\boldsymbol{m}}
\newcommand{\bst}{\boldsymbol{t}}
\newcommand{\bsalpha}{\boldsymbol{\alpha}}

\newcommand{\bszero}{\boldsymbol{0}}
\newcommand{\bfn}{\mathbf{n}}
\newcommand{\bfl}{\mathbf{l}}
\newcommand{\bfk}{\mathbf{k}}
\newcommand{\bfx}{\mathbf{x}}

\newcommand{\RR}{\mathbb{R}}
\newcommand{\QQ}{\mathbb{Q}}
\newcommand{\ZZ}{\mathbb{Z}}
\newcommand{\NN}{\mathbb{N}}
\newcommand{\FF}{\mathbb{F}}
\newcommand{\CC}{\mathbb{C}}
\newcommand{\EE}{\mathbb{E}}
\newcommand{\DD}{\mathbb{D}}
 
\newcommand{\walb}{\,_b{\rm wal}}
\newcommand{\icomp}{\mathtt{i}}

\newcommand{\Dcal}{\mathcal{D}}
\newcommand{\cH}{\mathcal{H}}
\newcommand{\cP}{\mathcal{P}}
\newcommand{\cZ}{\mathcal{Z}}
\newcommand{\cS}{\mathcal{S}}
\newcommand{\uu}{\mathfrak{u}}
\newcommand{\dint}{\,{\rm d}}

\newtheorem{theorem}{Theorem}[section]
\newtheorem{lemma}[theorem]{Lemma}
\newtheorem{proposition}[theorem]{Proposition}

\newtheorem{definition}[theorem]{Definition}

\newcommand{\satop}[2]{\stackrel{\scriptstyle{#1}}{\scriptstyle{#2}}}

\makeatletter
\newcommand{\rdots}{\mathinner{\mkern1mu\lower-1\p@\vbox{\kern7\p@\hbox{.}}
\mkern2mu \raise4\p@\hbox{.}\mkern2mu\raise7\p@\hbox{.}\mkern1mu}} \makeatother

\newenvironment{proof}{\begin{trivlist}
    \item[\hskip\labelsep{\it Proof.}]}{$\hfill\Box$\end{trivlist}}

\title{Proof Techniques in Quasi-Monte Carlo Theory}
\author{Josef Dick\thanks{J.~Dick is supported by an Australian Research Council QEII Fellowship.}, Aicke Hinrichs, Friedrich Pillichshammer\thanks{F. Pillichshammer is supported by the Austrian Science Fund (FWF): Project F5509-N26, which is a part of the Special Research Program "Quasi-Monte Carlo Methods: Theory and Applications".}}

\begin{document}

\maketitle 
\label{firstpage}

\begin{abstract}
In this survey paper we discuss some tools and methods which are of use in quasi-Monte Carlo (QMC) theory. We group them in chapters on Numerical Analysis, Harmonic Analysis, Algebra and Number Theory, and Probability Theory. We do not provide a comprehensive survey of all tools, but focus on a few of them, including reproducing and covariance kernels, Littlewood-Paley theory, Riesz products, Minkowski's fundamental theorem, exponential sums, diophantine approximation, Hoeffding's inequality and empirical processes, as well as other tools. We illustrate the use of these methods in QMC using examples.
\end{abstract}

\tableofcontents


\section{Introduction}\label{sec_intro}

Quasi-Monte Carlo (QMC) rules are quadrature rules which can be used to approximate integrals defined on the $s$-dimensional unit cube $[0,1]^s$
\begin{equation*}
\int_{[0,1]^s} f(\bsx) \,\mathrm{d} \bsx \approx \frac{1}{N} \sum_{n=0}^{N-1} f(\bsx_n),
\end{equation*}
where $\cP=\{\bsx_0,\bsx_1,\ldots,\bsx_{N-1}\}$ are deterministically chosen quadrature points in $[0,1)^s$.
In QMC theory one is interested in a number of questions. Of importance is the integration error
\begin{equation*}
\left| \int_{[0,1]^s} f(\bsx) \,\mathrm{d} \bsx - \frac{1}{N} \sum_{n=0}^{N-1} f(\bsx_n) \right|
\end{equation*}
and how it behaves as $N$ and/or $s$ increases. Various settings can be defined to analyze this error. For instance, one can consider the worst-case error: Here one uses a Banach space $(\cH, \|\cdot\|)$ and considers
\begin{equation*}
{\rm wce}(\cH,\cP)=\sup_{\satop{f \in \cH}{\|f\| \le 1}} \left| \int_{[0,1]^s} f(\bsx) \,\mathrm{d} \bsx - \frac{1}{N} \sum_{n=0}^{N-1} f(\bsx_n) \right|.
\end{equation*}
Particularly nice examples of such function spaces are so-called reproducing kernel Hilbert spaces. We review essential properties of reproducing kernel Hilbert spaces in Section~\ref{sec_numer_anal}. Other settings include the average case error: In this case one defines a probability measure $\mu$ on the function space $\cH$ and then studies the expectation value of the integration error
\begin{equation*}
{\rm ace}_p(\cH,\cP)=\left( \mathbb{E} \left| \int_{[0,1]^s} f(\bsx) \,\mathrm{d} \bsx - \frac{1}{N} \sum_{n=0}^{N-1} f(\bsx_n) \right|^p \right)^{1/p}.
\end{equation*}
Such an investigation can be carried out with the help of covariance kernels. There are a number of relations to reproducing kernels, which we also discuss in Section~\ref{sec_numer_anal}.

Covariance kernels also appear in stochastic processes, which themselves are important in applications in financial mathematics and partial differential equations (PDEs) with random coefficients, for instance. We discuss all these connections in the section on numerical analysis, Section~\ref{sec_numer_anal}, in which we also treat some further useful tools, like the use of bump functions to prove lower bounds and the Rader transform. Also the connection between the integration error and discrepancy of the quadrature points is shown and the Koksma-Hlawka inequality is described in this context.

The analysis of the integration error is often greatly helped by using orthogonal expansions. These can be Fourier series, Walsh series or Haar series for instance. Tools from harmonic analysis are important here. For instance the proof of strong lower bounds is facilitated by using the Littlewood-Paley inequality and Riesz products. We devote a section on harmonic analysis (Section~\ref{sec_har_anal}) to this topic to give the reader an idea of how those methods are applied in QMC.

Another important topic in QMC is the construction of good quadrature points which can be used in computation. This area makes fundamental use of algebra and number theory. Finite fields, characters and duality theory are of importance here, as well as a number of other topics including exponential sums, $b$-adic numbers, and diophantine approximation. These tools are reviewed and illustrated in the context of QMC in Section~\ref{sec_alg_num}.

Although many useful explicit constructions are known based on algebraic and number theoretic methods, in some instance one can show stronger results by switching to methods which only prove the existence of some point sets, rather than explicit constructions. The simplest instance of proving an existence result can be illustrated by the principle that for a given set of real numbers $a_1, a_2, \ldots, a_N$, at least one of those numbers is bounded above by the average $\frac{1}{N} \sum_{n=1}^N a_n$. This can be rephrased in terms of random variables and expectation values and leads to the probabilistic method. There are a number of sophisticated tools available from this area which go much further than the simple averaging argument described above, for instance Hoeffding's inequality, VC-classes and empirical processes. These methods are illustrated in Section~\ref{sec_prob}, which is devoted to the use of probability theory in QMC.

This article does not provide an introduction to QMC theory per se. The main goal is to illustrate the use of the tools mentioned above in QMC theory via some examples. The results in QMC theory which we use to illustrate these ideas are not always the most interesting cases since the emphasis is mainly on the tools and not the QMC results. Often we use results from QMC theory which highlight the concepts from the areas of numerical analysis, harmonic analysis, algebra and number theory and probability theory, and not the particular results from QMC theory. 

The motivation for the approach taken in this paper lies in the fact that introductions to various aspects of QMC theory have already appeared in a number of monographs and major survey articles in recent years. We mention those which are in preparation, to appear or appeared in the last ten years at the writing of this paper in chronological order. Strauch and Porubsk{\'y}~\cite{SP} provide a sampler of results on the distribution of sequences. This book includes many of the older results which are not included in other publications. The series of monographs \cite{NW08, NW10, NW12} by Novak and Wo\'zniakowski is devoted to Information-Based Complexity. QMC plays some role in there since it can be used to show tractability results in high dimensional integration problems. These monographs also provide the necessary background on various settings, from function spaces to different error criteria, which can be used to study QMC methods. Lemieux's work \cite{Lem} discusses Monte Carlo methods, including pseudo 
random number 
generation, QMC and Markov chain Monte Carlo, and various aspects of their use in applications. The monograph \cite{DP10} by Dick and Pillichshammer studies digital nets and sequences. These point sets and sequences can be used in QMC integration. Results on numerical integration and their connection to discrepancy theory are also explained in there. Triebel~\cite{tri2010, tri2014} studies connections of discrepancy theory and numerical integration via the study of function spaces. Another introductory book on Monte Carlo methods is by M\"uller-Gronbach, Novak and Ritter~\cite{MGNR} (in German). It discusses algorithmic aspects, simulation techniques, variance reduction, Markov chain Monte Carlo and numerical integration. The survey article \cite{DKS13} by Dick, Kuo and Sloan focuses on high dimensional numerical integration using QMC rules. Numerical integration in infinite dimensional spaces is also briefly discussed. The textbook \cite{LP} by Leobacher and Pillichshammer provides an introduction to 
QMC theory and discusses applications to various areas. A number of articles covering various aspects of discrepancy theory is provided in the  monograph \cite{CST}, edited by Chen, Srivastav and Travaglini. One of those articles relates discrepancy theory to QMC methods and shows how various parts of discrepancy theory can be used in QMC theory. Also deep results on discrepancy theory are discussed in various articles. Kritzer, Niederreiter, Pillichshammer and Winterhof \cite{KNPW14} are editors of a further book consisting of survey articles focusing on number theoretic constructions of point sets and sequences, uniform distribution theory, and quasi-Monte Carlo methods. Owen~\cite{Owen} is preparing a comprehensive introduction to Monte Carlo methods covering anything from Monte Carlo, quasi-Monte Carlo to Markov chain Monte Carlo, non-uniform random number generation, variance reduction and importance sampling as well as other aspects.

Given that many aspects of QMC theory have been surveyed  or covered in textbooks and research monographs, we aim to provide a survey of proof techniques and tools which are used in QMC theory. Although these tools often appear as part of proofs of theorems in QMC theory, they have usually not been the focus themselves in these other works. We do so here by introducing various methods and illustrating them via examples. 

\section{Numerical Analysis}\label{sec_numer_anal}

Numerical integration is a classical topic in numerical analysis. The Koksma-Hlawka inequality is a basic result in QMC theory. Its establishment (1941 in dimension one by Koksma and 1961 in arbitrary dimension by Hlawka) can be considered as a starting point for the analysis of QMC methods. In the modern context, such inequalities can be considered as bounds for worst-case errors in reproducing kernel Hilbert spaces or more general function spaces. Thus reproducing kernel functions play a significant role in studying QMC methods. Reproducing kernel functions themselves have many similarities to covariance kernels. The latter are important when studying average case errors, or problems defined over random fields or stochastic processes. Stochastic processes are for instance used in financial mathematics to model the stock price, or in physical applications to model the permeability of porous media. These applications lead to stochastic differential equations and partial differential equations with 
random coefficients. In some of these applications, QMC is used successfully as a sampling technique to obtain estimations of the expectation value of, for instance, the payoff function of an option or a linear functional of a solution of a PDE. In the following we survey some of the essential tools in this area.

\subsection{Reproducing Kernel Hilbert Spaces}\label{sec_RKHS}

Reproducing kernel Hilbert spaces play a fundamental role in QMC theory nowadays. The basic reference for reproducing kernel Hilbert spaces is \cite{Aron50}. Since we consider QMC in this paper, we restrict the domain to the unit cube $[0,1]^s$. In the following let $\overline{a}$ denote the complex conjugate of a complex number $a \in \mathbb{C}$.

\begin{definition}
A function $K: [0,1]^s \times [0,1]^s \to \mathbb{C}$ is a reproducing kernel if
\begin{enumerate}
\item $K(\bsx, \bsy) = \overline{ K(\bsy, \bsx) }$ for all $\bsx, \bsy \in [0,1]^s$ (\emph{symmetry of $K$}), and
\item for all $a_1, a_2, \ldots, a_N \in \mathbb{C}$ and all $\bsx_1, \bsx_2,\ldots, \bsx_N \in [0,1]^s$ we have
\begin{equation*}
\sum_{n, m = 1}^N a_n \overline{a}_m K(\bsx_n, \bsx_m) \ge 0
\end{equation*}
(\emph{positive semi-definiteness of $K$}).
\end{enumerate}
\end{definition}

A reproducing kernel $K$ uniquely defines a space $\cH_K$ of functions on $[0,1]^s$ and an inner product $\langle \cdot, \cdot \rangle_K$ on $\cH_K$. The corresponding norm is denoted by $\|\cdot\|_K$. The following properties are equivalent to the symmetry and positive semi-definiteness above.
\begin{enumerate}
\item[i)] $K(\cdot, \bsy) \in \cH_K$ for each fixed $\bsy \in [0,1]^s$;
\item[ii)] $\langle f, K(\cdot, \bsy) \rangle_K = f(\bsy)$ for all $\bsy \in [0,1]^s$ and $f \in \cH_K$;
\item[iii)] if $L: [0,1]^s \times [0,1]^s \to \mathbb{R}$ satisfies i) and ii), then $L = K$.
\end{enumerate}

\subsubsection*{Examples: Reproducing kernel Hilbert spaces derived from expansions}

\begin{enumerate}
\item \emph{Polynomial space} \\ In the first example we consider a space $\cH$ of polynomials $f(x) = a_0 + a_1 x + \cdots + a_r x^r$  on the interval $[0,1]$ of degree at most $r$, where $a_i \in \mathbb{C}$. The basic functions are the monomials $x^i$, $0 \le i \le r$, and each polynomial can be represented as a linear combination of these functions. We can define an inner product for polynomials $f_i = a_{i,0} + a_{i,1} x + \cdots + a_{i,r} x^r$ by
\begin{equation*}
\langle f_1, f_2 \rangle_1 = \sum_{\ell=0}^r a_{1,\ell} \overline{a_{2, \ell}}.
\end{equation*}
With this inner product, the monomials $x^i$ are orthonormal, that is
\begin{equation*}
\langle x^i, x^j \rangle_1 = \delta_{i,j},
\end{equation*}
where $\delta_{i,j}$ is the Kronecker $\delta$-symbol.

The task now is to find a function $K_1(x,y):[0,1] \times [0,1] \to \mathbb{C}$ which satisfies the reproducing property $\langle f, K_1(\cdot, y)\rangle_1 = f(y)$. This function is given by $K_1(x,y) = 1 + x \overline{y} + x^2 \overline{y}^2 + \cdots + x^r \overline{y}^r$ as can easily be verified. (Since we assume that $x, y \in [0,1]$, we also have $K_1(x,y) = 1 + x y + x^2 y^2 + \cdots + x^r y^r$ and hence the kernel $K_1 : [0,1] \times [0,1] \to \mathbb{R}$ is actually real-valued.) 

An alternative way of defining an inner product on the space of polynomials of degree at most $r$ is the following approach.
For $i \in \mathbb{N}_0$ let $B_i$ denote the Bernoulli polynomial of degree $i$. Use the expansion $g(x) = b_0 B_0 + b_1 B_1(x) + \cdots + b_r B_r(x)$, where $b_i \in \mathbb{C}$. Again one obtains polynomials of degree at most $r$ this way. We can define the inner product
\begin{equation*}
\langle g_1, g_2 \rangle_2 = \sum_{\ell=0}^r b_{1, \ell} \overline{b_{2, \ell}},
\end{equation*}
for 
\begin{equation}\label{expansion_Bernoulli}
g_i(x) = b_{i,0} + b_{i,1} B_1(x) + \cdots + b_{i,r} B_r(x). 
\end{equation}
This inner product differs from the first case. In fact, now the Bernoulli polynomials are an orthonormal basis $\langle B_i, B_j \rangle_2 = \delta_{i,j}$. The reproducing kernel is now given by $K_2(x,y) = B_0(x) B_0(y) + B_1(x) B_1(y) + \cdots + B_{r}(x) B_r(y)$ (note that the coefficients of the Bernoulli polynomials are all real numbers, hence for $y \in [0,1]$ we have $B_k(y) = \overline{B_k(y)}$).

{\it Caution:} We provide an example where the above principles fail. Consider all polynomials of degree at most $1$ of the form 
\begin{equation}\label{expansion_poly_fail}
f_i(x) = a_{i,0} + a_{i,1} x + b_{i,1} B_1(x).
\end{equation}
One could define the inner product $\langle f_1, f_2 \rangle_3 = a_{1,0} \overline{a_{2,0}} + a_{1,1} \overline{ a_{2,1} } + b_{1,1} \overline{ b_{2,1} }$. However, this is not well defined, since in the expansion \eqref{expansion_poly_fail} the values of $a_{i,0}, a_{i,1}, b_{i,1}$ are not uniquely defined.

\item \emph{Korobov space} \\ This space is a space of Fourier series
\begin{equation*}
f(x) = \sum_{k \in \mathbb{Z}} \widehat{f}(k) \exp(2\pi \icomp k x),
\end{equation*}
where $\icomp=\sqrt{-1}$ and $\widehat{f}(k)=\int_0^1 f(x) \exp(-2 \pi \icomp k x) \, \mathrm{d} x$. For $\alpha > 1/2$ we define an inner product by
\begin{equation*}
\langle f, g \rangle_{K_\alpha} = \sum_{k \in \mathbb{Z}} \widehat{f}(k) \overline{\widehat{g}(k)} \max(1, |k|)^{2 \alpha}.
\end{equation*}
Its reproducing kernel $K_\alpha: [0,1] \times [0,1] \to \mathbb{C}$ is given by
\begin{equation*}
K_\alpha(x, y) = \sum_{k \in \mathbb{Z}} \max(1, |k|)^{- 2 \alpha} \exp(2\pi \icomp k (x-y)).
\end{equation*}
(In fact we have $K_\alpha(x, y) \in \mathbb{R}$ for all $x, y \in [0,1]$.)

\item \emph{Unanchored Sobolev space} \\ The unanchored Sobolev space is the direct sum of the Korobov space and the polynomial space using the Bernoulli expansion \eqref{expansion_Bernoulli}.

For $i=1,2$ let $h_i$ be a function in the Korobov space where $\alpha = 1$ such that $\int_0^1 h_i(x) \,\mathrm{d} x = 0$. Let
\begin{equation*}
f_i(x) = b_{i,0} B_0(x) + b_{i,1} B_1(x) + h_i(x) = b_{i,0} B_0(x) + b_{i,1} B_1(x) + \sum_{k \in \mathbb{Z} \setminus \{0\}} \widehat{h}_i(k) \exp(2\pi \icomp k x),
\end{equation*}
where $B_0(x) = 1$ and $B_1(x) = x-1/2$ are the Bernoulli polynomials. By assuming that $\int_0^1 h_i(x) \,\mathrm{d} x = 0$ this representation is unique, since the constant part is in $b_{i,0}$ and $B_1(x) = x-1/2$ is not in the Korobov space. We can define an inner product by
\begin{equation*}
\langle f_1, f_2 \rangle_K = b_{1,0} \overline{b_{2,0}} + b_{1,1} \overline{ b_{2,1} } + \frac{1}{(2\pi)^2}  \sum_{k \in \mathbb{Z} \setminus \{0\}} \widehat{h}_1(k) \overline{\widehat{h}_2(k)} \, |k|^2.
\end{equation*}
The role of the normalizing factor $(2\pi)^{-2}$ will soon become clear, but has otherwise no bearings on the principles used to define the inner product. The reproducing kernel is given by
\begin{equation*}
K(x,y) = B_0(x)  B_0(y) + B_1(x) B_1(y) + (2\pi)^2 \sum_{k \in \mathbb{Z}\setminus \{0\}} |k|^{-2} \exp(2\pi \icomp k (x-y)).
\end{equation*}

The representation above can be simplified. The inner product is given by
\begin{equation*}
\langle f_1, f_2 \rangle_K = \int_0^1 f_1(x) \,\mathrm{d} x \int_0^1 \overline{f_2(x)} \,\mathrm{d} x + \int_0^1 f_1'(x) \overline{f_2'(x)} \,\mathrm{d} x.
\end{equation*}

For $\ell \in \mathbb{N}$, the Bernoulli polynomial $B_\ell$ has the Fourier series expansion
\begin{equation*}
B_\ell(x) = - \frac{\ell!}{(2\pi \icomp)^\ell} \sum_{k \in \mathbb{Z}\setminus \{0\}} k^{-\ell} \exp(2\pi \icomp k x).
\end{equation*}
Thus we can write the reproducing kernel as $$K(x, y) = 1 + B_1(x) B_1(y) + \tfrac{1}{2} B_2(|x-y|).$$ 

This approach can be extended to smoothness $\alpha > 1$ with $\alpha \in \mathbb{N}$, by using the Korobov space of smoothness $\alpha$ and the space of Bernoulli polynomials of degree up to $\alpha$. (Note that the Bernoulli polynomials of degree $\ell \le \alpha$ are not in the Korobov space of smoothness $\alpha$, thus this approach is well defined.)

\item \emph{Anchored Sobolev space} \\ The anchored Sobolev space is based on the Taylor series expansion with integral remainder
\begin{equation*}
f(y) = f(0) + \int_0^1 f'(t) 1_{[0,y]}(t) \,\mathrm{d} t.
\end{equation*}
We define an inner product by
\begin{equation*}
\langle f, g \rangle_K = f(0) \overline{g(0)} + \int_0^1 f'(t) \overline{g'(t)} \,\mathrm{d} t.
\end{equation*}
The reproducing kernel is given by 
\begin{equation*}
K(x, y) = 1 + \int_0^1 1_{[0,x]}(t) 1_{[0,y]}(t) \,\mathrm{d} t =  1 + \min(x, y).
\end{equation*}

By using the same principle as above but with a Taylor series expansion with integral remainder involving derivatives up to order $r$, we obtain the anchored Sobolev space of order $r$.

\end{enumerate}

\medskip

To define $s$-variate function spaces we can use the $s$-fold tensor product $\cH \otimes \cH \otimes \cdots \otimes \cH$. The reproducing kernel is in this case given by the $s$-fold product of the one-dimensional reproducing kernels, i.e.,
\begin{equation*}
K(\bsx, \bsy) = \prod_{j=1}^s K(x_j, y_j).
\end{equation*}

\subsubsection*{An important property}

The following property, valid for any reproducing kernel Hilbert space, is frequently used in QMC theory. Let $T: \cH \to \mathbb{R}$ be a continuous linear functional. Then the order of inner product and linear functional can always be interchanged, that is
\begin{equation*}
T(\langle f, K(\cdot, \bsx) \rangle_K) = \langle f, T(K(\cdot, \bsx)) \rangle_K.
\end{equation*}
This follows from the Riesz representation theorem. A proof can be found in \cite[Section~2.3.3]{DP10}.

\subsubsection*{Reproducing kernels and the worst-case error}

The {\it worst-case integration error} of a {\it QMC rule} $$\frac{1}{N} \sum_{n=0}^{N-1} f(\bsx_n) \ \ \ \mbox{ for } \ f \in \cH$$ based on a point set $\cP=\{\bsx_0,\ldots,\bsx_{N-1}\}$ over a certain function space $\cH$ with norm $\| \cdot \|$  is an important tool for assessing the quality of the quadrature point set. It is defined as
\begin{equation*}
\mathrm{wce}(\cH, \cP) = \sup_{\satop{f \in \cH}{\|f\|} \le 1} \left|\int_{[0,1]^s} f(\bsx) \,\mathrm{d} \bsx - \frac{1}{N} \sum_{n=0}^{N-1} f(\bsx_n) \right|.
\end{equation*}

If $\cH=\cH_K$ is a reproducing kernel Hilbert space, then the worst-case error can be stated explicitly in terms of the reproducing kernel $K$. Indeed we have for any $f \in \cH_K$ that
\begin{align}\label{int_rep_kernel}
\int_{[0,1]^s} f(\bsx) \,\mathrm{d} \bsx - \frac{1}{N} \sum_{n=0}^{N-1} f(\bsx_n)  = &  \int_{[0,1]^s} \langle f, K(\cdot, \bsx) \rangle_K \,\mathrm{d} \bsx - \frac{1}{N} \sum_{n=0}^{N-1} \langle f, K(\cdot, \bsx_n) \rangle_K \nonumber  \\ = &   \left\langle f,  \int_{[0,1]^s} K(\cdot, \bsx)  - \frac{1}{N} \sum_{n=0}^{N-1}  K(\cdot, \bsx_n) \right\rangle_K = \langle f, h \rangle_K,
\end{align}
where
\begin{equation}\label{def_h}
h(\bsy) = \int_{[0,1]^s} K(\bsy, \bsx)\, {\rm d} \bsx - \frac{1}{N} \sum_{n=0}^{N-1}  K(\bsy, \bsx_n).
\end{equation}
Thus, for any function $f \in \cH_K$ with $f \neq 0$ we have
\begin{equation*}
\frac{1}{\|f\|_K} \left| \int_{[0,1]^s} f(\bsx) \,\mathrm{d} \bsx - \frac{1}{N} \sum_{n=0}^{N-1} f(\bsx_n) \right| \le \|h\|_K.
\end{equation*}
On the other hand, we can achieve equality by considering the integration error of the function $h$. Thus we obtain that
\begin{align}
\mathrm{wce}^2(\cH_K, \cP) & = \|h\|_K^2  = \langle h, h \rangle_{K} \label{wce_h} \\ = & \int_{[0,1]^s} \int_{[0,1]^s} K(\bsx, \bsy) \,\mathrm{d} \bsx \,\mathrm{d} \bsy - \frac{2}{N} \sum_{n=0}^{N-1} \int_{[0,1]^s} K(\bsx, \bsx_n) \,\mathrm{d} \bsx + \frac{1}{N^2} \sum_{n, n'=0}^{N-1} K(\bsx_n, \bsx_{n'}). \label{wce_k}
\end{align}

\subsection{Koksma-Hlawka Inequality}\label{subsec_KH_ugl}

The Koksma-Hlawka inequality is a classic bound on the integration error of QMC rules. We give an example of this type of inequality using reproducing kernel Hilbert spaces.

We start by introducing the reproducing kernel $K:[0,1]\times [0,1] \to \mathbb{R}$ given by
\begin{equation*}
K(x,y) = 1 + \int_0^1 1_{[x,1]}(t) 1_{[y,1]}(t) \,\mathrm{d} t = 1 + \min(1-x, 1-y).
\end{equation*}
The inner product in the corresponding reproducing kernel Hilbert space is given by
\begin{equation*}
\langle f, g \rangle_K = f(1) g(1) + \int_0^1 f'(t) g'(t) \,\mathrm{d} t.
\end{equation*}
For dimensions $s > 1$ we use the kernel
\begin{equation*}
K(\bsx, \bsy) = \prod_{j=1}^s K(x_j, y_j).
\end{equation*}
Then the inner product is given by
\begin{equation*}
\langle f, g \rangle_K = \sum_{\uu \subseteq [s]} \int_{[0,1]^{|\uu|}} \frac{\partial^{\uu} f}{\partial \bsx_{\uu}}(\bsx_{\uu}; \boldsymbol{1}) \frac{\partial g}{\partial \bsx_{\uu}}(\bsx_{\uu}; \boldsymbol{1}) \,\mathrm{d} \bsx_{\uu},
\end{equation*}
where $[s]:=\{1,2,\ldots,s\}$ and where for $\uu \subseteq [s]$ and $\bsx=(x_1,x_2,\ldots,x_s)$ we write $\bsx_{\uu}=(x_j)_{j \in \uu}$ and $(\bsx_{\uu}; \boldsymbol{1})=(z_1,z_2,\ldots,z_s)$ with $$z_j=\left\{ 
\begin{array}{ll}
x_j & \mbox{ if } j \in \uu,\\
1 & \mbox{ if } j \not\in \uu. 
\end{array}\right.$$

Eq. \eqref{int_rep_kernel} provides a representation of the integration error in terms of the reproducing kernel. Of essence here is the function $h$, which for our specific reproducing kernel $K$ is given by
\begin{align*}
h(\bsy) = & \prod_{j=1}^s \int_0^1 K(y_j, x_j) \,\mathrm{d} x_j - \frac{1}{N} \sum_{n=0}^{N-1} \prod_{j=1}^s K(y_j, x_{n, j}) \\ = & \prod_{j=1}^s \frac{3-y_j^2}{2} - \frac{1}{N} \sum_{n=0}^{N-1} \prod_{j=1}^s \left(1 + \min(1-y_j, 1-x_{n,j}) \right).
\end{align*}

From \eqref{wce_h} we have that the worst-case error for integration in $\cH_K$ is given by $\|h\|_K$. We now compute this norm explicitly. To do so, we need the partial derivatives
\begin{align*}
\frac{\partial^{\uu} h}{\partial \bsy_{\uu}}(\bsy_{\uu}; \boldsymbol{1}) = & (-1)^{|\uu|+1} \left( \frac{1}{N} \sum_{\bsz\in{\cal P}} {\mathbf 1}_{B_{(\bsy_{\uu}, \boldsymbol{1}) }}(\bsz) -  \, {\rm vol} (B_{(\bsy_{\uu}, \boldsymbol{1}) })\right).
\end{align*}
Here ${\rm vol} (B_{\bsy}) = y_1 \cdots y_s$ denotes the volume of the rectangular box $B_{\bsy}=[0, y_1) \times \ldots \times[0,y_s)$ for $\bsy=(y_1,\ldots,y_s)\in [0,1]^s$ and ${\mathbf 1}_{B_{\bsy}}$ is the characteristic function of the box $B_{\bsy}$.

Thus \eqref{wce_h} implies that
\begin{equation}\label{wce_L2dis}
\mathrm{wce}(\cH_K, \cP) = \left(\sum_{\uu \subseteq [s]} \int_{[0,1]^{|\uu|}} \left( \frac{1}{N} \sum_{\bsz\in{\cal P}} {\mathbf 1}_{B_{(\bsy_{\uu},1)}}(\bsz) -  \, {\rm vol} (B_{(\bsy_{\uu},1)})\right)^2 \,\mathrm{d} \bsy_{\uu} \right)^{1/2}.
\end{equation}

For an $N$-element point set ${\cal P}$ in the $s$-dimensional unit cube $[0,1)^s$ the {\em discrepancy function} 
$D_N$ is defined as
\begin{equation}
\label{eq:_disc}
   D_N({\cal P}, \bsy) := \frac{1}{N} \sum_{\bsz\in{\cal P}} {\mathbf 1}_{B_{\bsy}}(\bsz) -  \, {\rm vol} (B_{\bsy}). 
\end{equation}
The sum in the discrepancy function counts the number of points of ${\cal P}$ contained in $B_{\bsy}$ and the
discrepancy function measures the deviation of this number from the fair number of points $N \, {\rm vol} (B_{\bsy})$ which
would be achieved by a perfect (but impossible) uniform distribution of the points of ${\cal P}$. 

Since \eqref{wce_L2dis} is the $L_2$ norm of the discrepancy function, it is also called the $L_2$ discrepancy of the point set $\cP$. 
The $L_p$ version of the discrepancy function also makes sense and can also be motivated by numerical integration. Using \eqref{int_rep_kernel} we have
\begin{align*}
\int_{[0,1]^s} f(\bsx) \,\mathrm{d} \bsx - \frac{1}{N} \sum_{n=0}^{N-1} f(\bsx_n) = & \langle f, h \rangle_K \\ = & \sum_{\uu \subseteq [s]} (-1)^{|\uu|+1} \int_{[0,1]^{|\uu|}} \frac{\partial^{\uu} f}{\partial \bsx_{\uu}}(\bsx_{\uu}; \boldsymbol{1}) D_N(\cP, (\bsx_{\uu}; \boldsymbol{1})) \,\mathrm{d} \bsx_{\uu}.
\end{align*}
Taking the absolute value and applying H\"older's inequality for integrals and sums, we obtain that
\begin{equation}\label{KH_ugl}
\left|\int_{[0,1]^s} f(\bsx) \,\mathrm{d} \bsx - \frac{1}{N} \sum_{n=0}^{N-1} f(\bsx_n) \right| \le L_{p,q}(\cP) \|f\|_{p',q'},
\end{equation}
where $1/p + 1/p' = 1$ and $1/q + 1/q'=1$,
\begin{equation*}
L_{p,q}(\cP) = \left(\sum_{\uu \subseteq [s]} \left( \int_{[0,1]^{|\uu|}} \left( D_N(\cP,(\bsy_{\uu};\boldsymbol{1}))\right)^{q} \,\mathrm{d} \bsy_{\uu} \right)^{p/q} \right)^{1/p}
\end{equation*}
and
\begin{equation*}
\|f\|_{p',q'}=  \left(\sum_{\uu \subseteq [s]} \left(\int_{[0,1]^{|\uu|}} \left|\frac{\partial^{\uu} f}{\partial \bsx_{\uu}}(\bsx_{\uu}; \boldsymbol{1}) \right|^{q'} \,\mathrm{d} \bsx_{\uu} \right)^{p'/q'} \right)^{1/p'},
\end{equation*}
with the obvious modifications if $p, p', q$ or $q'$ are $\infty$. The error estimate \eqref{KH_ugl} is called a {\it Koksma-Hlawka inequality}. In its classical form it uses $q=p=\infty$ and the variation of $f$ in the sense of Hardy and Krause instead of the norm $\|f\|_{1,1}$ (see, e.g., \cite{kuinie}).

Let $\cH$ be a normed function space which contains the discrepancy function of any point set. 
Then we denote the norm of the discrepancy function $D_N({\cal P}, \,\cdot\,)$, as defined in \eqref{eq:_disc}, by $D_N({\cal P}, \cH)$.
For $0<p\le\infty$, the (quasi)-norm $D_N({\cal P}, L_p)$ is called the {\it $L_p$-discrepancy} of the point set ${\cal P}$.
In particular, we abbreviate $$D_N({\cal P}) = D_N({\cal P}, L_\infty) = \sup_{\bsx\in [0,1]^s} \big| D_N({\cal P}, x) \big|,$$ which is often called the {\it star discrepancy} of $\cP$.

\subsection{Jensen's Inequality}

Jensen's inequality is an important tool to show improved error convergence rates.

\begin{theorem}[Jensen's inequality]\label{jens}
For any $\lambda \in (0,1]$ and nonnegative reals $a_k$ we have  
\begin{equation*}
\left(\sum_k a_k\right)^{\lambda} \le \sum_k a_k^{\lambda}.
\end{equation*} 
\end{theorem}

\begin{proof}
We have $0 \le a_j/\left(\sum_{k} a_k\right) \le 1$ and hence, since $\lambda \in (0,1]$, we have $$\frac{a_j}{\sum_k a_k} \le \left(\frac{a_j}{\sum_k a_k}\right)^{\lambda}.$$ Summation over all $j$ implies $$1= \frac{\sum_j a_j}{\sum_k a_k} \le \frac{\sum_j a_j^\lambda}{\left(\sum_k a_k \right)^{\lambda}}$$ which finally yields the result.
\end{proof} 

As an example we consider the worst-case error of lattice rules in the Korobov space $\cH_{K_\alpha}$ with smoothness parameter $\alpha>1/2$ (see Section~\ref{sec_RKHS}). Let $N$ be a prime number. For an $N$-element lattice rule with generating vector $\bsg \in \ZZ^s$ (see Section~\ref{secLPS}) it can be shown that the worst-case error is given by
\begin{equation}\label{wce_sum}
{\rm wce}^2(\cH_{K_\alpha},\cP(\boldsymbol{g},N)) = \sum_{\boldsymbol{h} \in L^\perp \setminus \{\boldsymbol{0}\}} r_\alpha(\boldsymbol{h}),
\end{equation}
where $$L^\bot = \{ \boldsymbol{h} \in \mathbb{Z}^s: \boldsymbol{h} \cdot \boldsymbol{g} \equiv 0 \pmod{N}\}$$ is the so-called dual lattice (see also \eqref{def_duallatt_rank1} in Section~\ref{grchardual}), $r_\alpha(\bsh)=\prod_{j=1}^s r_\alpha(h_j)$, for $\bsh=(h_1,\ldots,h_s) \in  \mathbb{Z}^s$ and where for $h \in \mathbb{Z}$
\begin{equation*}
r_\alpha(h) = \left\{ 
\begin{array}{ll}
1 & \mbox{ if }  h=0,\\
|h|^{-2 \alpha} & \mbox{ if } h \not=0.
\end{array}
\right. 
\end{equation*}

A simple principle in showing the existence of a mathematical object with a certain property is to prove a bound on the average and then to conclude that there is at least one instance which is at least as good as average. In our context we average the squared worst-case error ${\rm wce}^2(\cH_{K_\alpha},\cP(\boldsymbol{g},N))$ over all lattice rules from a certain finite set of lattice rules and deduce that there must exist at least one lattice rule for which the squared worst-case error is as good as the upper bound on this average. Let $G_N=\{1,2,\ldots,N-1\}$. Then it can be shown that
\begin{align}\label{av_wc_kor}
\frac{1}{(N-1)^s} \sum_{\boldsymbol{g} \in G_N^s} {\rm wce}^2(\cH_{K_\alpha},\cP(\boldsymbol{g},N)) \le \frac{(1+2 \zeta(2 \alpha))^s}{N-1},
\end{align}
where $\zeta(x)=\sum_{j=1}^{\infty} j^{-x}$ is the Riemann zeta function. Hence there must exist a generating vector $\bsg_\ast \in G_N^s$ which satisfies $${\rm wce}^2(\cH_{K_\alpha},\cP(\boldsymbol{g}_\ast,N)) \le \frac{(1+2 \zeta(2 \alpha))^s}{N-1}.$$ This yields a convergence rate for the worst-case error in $\cH_{K_\alpha}$ of order $O(N^{-1/2})$. The problem with this bound is that it does not reflect the smoothness $\alpha$ of the considered function space. This problem can be overcome with the help of Jensen's inequality which, by applying it to \eqref{wce_sum}, implies that
\begin{equation}\label{e_Jensen}
[{\rm wce}^2(\cH_{K_\alpha},\cP(\boldsymbol{g},N))]^\lambda \le {\rm wce}^2(\cH_{K_{\alpha \lambda}},\cP(\boldsymbol{g},N)) \ \ \ \mbox{ for all } \ \frac{1}{2\alpha} < \lambda \le 1,
\end{equation}
where the restriction $\frac{1}{2\alpha} < \lambda$ is added to ensure that ${\rm wce}^2(\cH_{K_{\alpha \lambda}},\cP(\boldsymbol{g},N))$ is finite.

Now we can apply the same averaging principle as above to ${\rm wce}^2(\cH_{K_{\alpha \lambda}},\cP(\boldsymbol{g},N))$. This implies for given $\frac{1}{2\alpha} < \lambda \le 1$ the existence of a generating vector $\bsg_\ast \in G_N^s$ which satisfies $${\rm wce}^2(\cH_{K_{\alpha \lambda}},\cP(\boldsymbol{g}_\ast,N)) \le \frac{(1+2 \zeta(2 \alpha \lambda))^s}{N-1}.$$ Inserting this result into \eqref{e_Jensen} one obtains the existence of a generating vector $\bsg_\ast \in G_N^s$ which satisfies $${\rm wce}^2(\cH_{K_\alpha},\cP(\boldsymbol{g}_\ast,N)) \le \frac{(1+2 \zeta(2 \alpha \lambda))^{s/\lambda}}{(N-1)^{1/\lambda}}.$$ In fact, let $\bsg_\ast$ be the generating vector which minimizes the worst-case error, that is, $${\rm wce}^2(\cH_{K_\alpha},\cP(\boldsymbol{g}_\ast,N)) = \min_{\bsg \in G_N^s} {\rm wce}^2(\cH_{K_\alpha},\cP(\boldsymbol{g},N)).$$ Then $${\rm wce}^2(\cH_{K_\alpha},\cP(\boldsymbol{g}_\ast,N)) \le \frac{(1+2 \zeta(2 \alpha \lambda))^{s/\lambda}}{(N-1)^{1/\lambda}} \quad \mbox{for all } \lambda \in \left(\tfrac{1}{2\alpha},  1\right],$$ i.e. $\boldsymbol{g}_\ast$ does not depend on $\lambda$. Since $\lambda$ can be chosen arbitrary close to $\frac{1}{2\alpha}$ this leads to an improved convergence rate for the worst-case error in $\cH_{K_\alpha}$ of order $O(N^{-\alpha+\varepsilon})$ for arbitrary small $\varepsilon>0$.

We point out that Jensen's inequality holds more generally for concave functions (or convex functions in the opposite direction). Let $\phi(x) = x^\lambda$ with $0 < \lambda \le 1$, then we can write the above form of Jensen's inequality as
\begin{equation}\label{Jensen_gen}
\phi\left(\sum_k a_k \right) \le \sum_k \phi(a_k).
\end{equation}
In particular, \eqref{Jensen_gen} holds also for concave functions $\phi:[0,\infty) \to [0, \infty)$. Considering the above example of the Korobov space, the general version of Jensen's inequality can be useful, for instance, when the endpoint $\lambda = \frac{1}{2\alpha}$ is of interest; i.e. when one aims for a bound on the worst-case error of the form $N^{-\alpha} (\log N)^c$, for some $c > 0$. In this case one may use $\phi$ which maps $|h|^{-2\alpha}$ to $|h|^{-1} (\log  2|h| )^{-\lambda}$ for some suitable choice of $\lambda > 1$. Or consider the case where $r_\alpha(h) = |h|^{-1} (\log 2 |h|)^\alpha$ for $h \neq 0$ and $\alpha > 1$. Then the corresponding Korobov space is still well defined, however, the inequality in Theorem~\ref{jens} cannot be used to yield an improved rate of convergence. However a different choice of $\phi$ does yield a convergence rate beyond $O(N^{-1/2})$. This case has been studied in \cite{DKLP14}.

\subsection{Mercer's Theorem}\label{subsec_Mercer}

In the examples of reproducing kernel Hilbert spaces we have seen that some expansions of functions (polynomials or Fourier series for instance) yield reproducing kernel Hilbert spaces in a natural way. One may ask whether such expansions exist for any reproducing kernel (i.e. any symmetric and positive semi-definite function). An affirmative answer to this question for continuous reproducing kernels is given by Mercer's theorem.

Let $K: [0,1]^s \times [0,1]^s \to \mathbb{C}$ be a reproducing kernel. Assume that $K$ is continuous. We define the linear operator $T_K: L_2([0,1]^s) \to L_2([0,1]^s)$ by
\begin{equation*}
T_K(f)(\bsx) = \int_{[0,1]^s} K(\bsx, \bsy) f(\bsy) \,\mathrm{d} \bsy.
\end{equation*}
Then $T_K$ is a self-adjoint, positive, compact operator on $L_2([0,1]^s)$. In the following we state a version of Mercer's theorem~\cite{Mercer} which we adapt to our situation.
\begin{theorem}[Mercer]
Let the reproducing kernel $K:[0,1]^s \times [0,1]^s \to \mathbb{C}$ be a continuous function. Then there exists a sequence of $L_2$ orthonormal eigenfunctions $\psi_{\ell}: [0,1]^s \to \mathbb{C}$, $\ell \in \mathbb{N}$, with corresponding nonnegative eigenvalues $(\lambda_\ell)_{\ell=1}^\infty$ of the operator $T_K$
\begin{equation*}
T_K (\psi_\ell)(\bsx) = \lambda_\ell \psi_\ell(\bsx) \quad \mbox{for all } \ell \in \mathbb{N}.
\end{equation*}
The reproducing kernel $K$ has the representation
\begin{equation*}
K(\bsx, \bsy) = \sum_{\ell=1}^\infty \lambda_\ell \psi_\ell(\bsx) \overline{\psi_\ell(\bsy)}.
\end{equation*}
\end{theorem}

\subsubsection*{Examples}

We now show some examples of reproducing kernels and their expansions. We have already seen an example where the eigenvalues and eigenfunctions are obvious:
\begin{enumerate}
\item \emph{Korobov space} \\
The reproducing kernel is given by $K_{\alpha}(x,y) = \sum_{k \in \mathbb{Z}} \max(1, |k|)^{-2\alpha} \exp(2\pi \icomp k (x-y))$; here the eigenvalues are $(\max(1, |k|)^{-2\alpha} )_{k \in \mathbb{Z}}$ and the eigenfunctions are $\exp(2\pi \icomp k x)$ for $k \in \mathbb{Z}$.

\medskip
\hspace{-1cm}
We consider now the unanchored and anchored Sobolev spaces.

\item \emph{Unanchored Sobolev space} \\ The eigenvalues and eigenfunctions of the reproducing kernel $K(x,y) = 1 + B_1(x) B_1(y) + \tfrac{1}{2} B_2(|x-y|)$ have been found in \cite{DNP}. The eigenvalues are $1, \pi^{-2}, (2\pi)^{-2}, (3\pi)^{-2}, \ldots$ and the eigenfunctions are $1, \sqrt{2} \cos(\pi x), \sqrt{2} \cos(2\pi x), \sqrt{2} \cos(3\pi x), \ldots$.

\item \emph{Anchored Sobolev space} \\ The eigenvalues and eigenfunctions of the reproducing kernel $K(x,y) = 1 + \min(x, y)$ have been found in \cite{WW99}. The eigenvalues are $\lambda_\ell = \alpha_\ell^{-2}$ for all $\ell \in \mathbb{N}$, where $\alpha_\ell \in ( (\ell-1) \pi, \ell \pi)$ is the unique solution of the equation
\begin{equation*}
\tan \alpha_\ell = \frac{1}{\alpha_\ell}.
\end{equation*}
\end{enumerate}

\subsubsection*{Example of the derivation of eigenvalues and eigenfunctions}

We consider another related example where we derive the eigenvalues and eigenfunctions via a solution to an ODE. Namely, consider the function
\begin{equation}\label{kernel_W}
K(x,y) = \min(x, y).
\end{equation}
This function is symmetric and positive semi-definite and therefore a reproducing kernel. We are interested in obtaining the eigenvalues and eigenfunctions of the operator
\begin{equation*}
T_K(f)(x) = \int_0^1 K(x,y) f(y) \,\mathrm{d} y = \int_0^1 \min(x, y) f(y) \,\mathrm{d} y.
\end{equation*}
Let $\lambda_\ell$ be an eigenvalue and $\psi_\ell$ the corresponding eigenfunction. Then
\begin{equation*}
\lambda_\ell \psi_\ell(x) = \int_0^1 \min(x, y) \psi_\ell(y) \,\mathrm{d} y = \int_0^x y \psi_\ell(y) \,\mathrm{d} y + \int_x^1 x \psi_\ell(y) \,\mathrm{d} y.
\end{equation*}
By setting $x = 0$ we obtain
\begin{equation*}
 \lambda_\ell \psi_\ell(0) = 0.
\end{equation*}
By differentiating with respect to $x$ we obtain
\begin{equation*}
\lambda_\ell \psi'_\ell(x) = \int_x^1 \psi_\ell(y) \,\mathrm{d} y.
\end{equation*}
Setting $x = 1$ in the above equation yields $\lambda_\ell \psi'_\ell(1) = 0$. By twice differentiating with respect to $x$ we obtain
\begin{equation*}
\lambda_\ell \psi_\ell''(x)= - \psi_\ell(x).
\end{equation*}
The function $\psi_\ell$ which satisfies the two boundary conditions and the last ODE is given by
\begin{equation*}
\psi_\ell(x) = \sqrt{2} \sin\left(\left(\ell - \frac{1}{2} \right) \pi x \right)
\end{equation*}
with corresponding eigenvalue $$\lambda_\ell = \left(\left(\ell-\frac{1}{2} \right) \pi\right)^{-2}$$ for $\ell \in \mathbb{N}$. The normalizing factor $\sqrt{2}$ is introduced, such that the functions $\psi_\ell$ are $L_2$ orthonormal.

Thus the reproducing kernel \eqref{kernel_W} can be written as
\begin{equation*}
K(x,y) = \sum_{\ell=1}^\infty \frac{\sqrt{2} \sin((\ell-1/2) \pi x)}{(\ell-1/2)\pi} \frac{\sqrt{2} \sin((\ell-1/2) \pi y)}{(\ell-1/2) \pi}.
\end{equation*}
Functions $f_i$ in the corresponding reproducing kernel Hilbert space $\cH_K$ have an expansion of the form
\begin{equation}\label{fi_Wiener}
f_i(x) = \sum_{\ell=1}^\infty \widehat{f}_i(\ell) \sqrt{2} \sin((\ell-1/2) \pi x),
\end{equation}
where
\begin{equation*}
\widehat{f}_i(\ell) = \int_0^1 f_i(x) \sqrt{2} \sin ( (\ell - 1/2) \pi x) \, {\rm d} x,
\end{equation*}
and the inner product is given by
\begin{equation*}
\langle f_1, f_2 \rangle_K = \sum_{\ell=1}^\infty \widehat{f}_1(\ell) \overline{\widehat{f}_2(\ell)} (\ell-1/2)^2 \pi^2.
\end{equation*}

\subsection{Covariance Kernel}

The covariance kernel has many similarities with the reproducing kernel. We restrict ourselves again to the domain $[0,1]^s$. A {\it covariance kernel} $C:[0,1]^s \times [0,1]^s \to \mathbb{R}$ is again a symmetric and positive semi-definite function (and is therefore also a reproducing kernel).

In QMC theory, the covariance kernel has two different uses. One is the study of the so-called average-case error and the other appears in the study of PDEs with random coefficients, where the covariance kernel describes the underlying random coefficients (or random field). These two cases are based on different interpretations of the covariance kernel.


\begin{enumerate}
\item \emph{Random function} \\ Let $\cH$ be a function class defined on $[0,1]^s$ and $\mathcal{B}(\cH)$ be a $\sigma$ algebra on $\cH$. Further let $\mu$ be a probability measure defined on $(\cH, \mathcal{B}(\cH))$. Then we define the covariance kernel $C:[0,1]^s \times [0,1]^s \to \mathbb{R}$ by
\begin{equation*}
C(\bsx, \bsy) = \int_{\cH} f(\bsx) f(\bsy) \, \mu(\mathrm{d} f).
\end{equation*}
That is, the covariance kernel is the expectation value over all functions in the class $\cH$ evaluated at the points $\bsx$ and $\bsy$. The functions $f \in \cH$ themselves are not random variables, but we choose $f \in \cH$ randomly, i.e., once a function $f$ is chosen it is entirely deterministic.

\item \emph{Stochastic processes and random fields}

In general, a stochastic process is a parameterized collection of random variables $\{X_{t}: t \in T\}$ defined on a probability space $(\Omega, \mathcal{F}, \mathbb{P})$ which assumes values in a measurable space $(S, \Sigma)$ and which is indexed by a totally ordered set $T$. A random field is also a parameterized collection of random variables $\{Z(\bsx): \bsx \in X\}$ defined on a probability space $(\Omega, \mathcal{F}, \mathbb{P})$ which assumes values in a measurable space $(S, \Sigma)$ but which is now indexed by a topological set $X$. As an intuitive guideline, in a stochastic process the parameter $t$ can be thought of as time, whereas in a random field one thinks of the parameter $\bsx$ as a location in space.

Here we will consider random fields $Z(\bsx)$, where $\bsx \in [0,1]^s$ and $S = \mathbb{R}$. Note that for $s=1$ $Z(x)$ can be thought of as a stochastic process or also a random field. The covariance kernel gives the covariance of the values of the random field $Z(\bsx)$ at the locations $\bsx, \bsy \in [0,1]^s$
\begin{equation*}
C(\bsx, \bsy) = \mathrm{cov}(Z(\bsx), Z(\bsy)).
\end{equation*}
For each value $\bsx$ in the domain $[0,1]^s$, the values $Z(\bsx)$ are random variables with some given distribution. More information on stochastic processes, martingales and stochastic differential equations can for instance be found in \cite{KS, RW1, RW2}.
\end{enumerate}

In the remainder of this subsection we deal with the first case of random functions.

\subsubsection*{Example: Continuous functions and the Wiener sheet measure}

A classic result in QMC theory is concerned with the average case error of the set of continuous functions which vanish at $0$ endowed with the Wiener sheet measure \cite{Woz91}. 

We give an example of how one can define a probability measure on a function space $\cH$. Let $\cH$ be the class of functions given by
\begin{equation*}
\cH = \left\{f:[0,1]\to\mathbb{R}: f(0) = 0, f \mbox{ is continuous} \right\}.
\end{equation*}

The functions $\left( \sqrt{2} \sin ((\ell-1/2) \pi x) \right)_{\ell \in \mathbb{N}}$ are $L_2$ orthonormal. First note that the functions in $\cH$ permit expansions of the form
\begin{equation}\label{exp_rand_fun}
f(x) = \sum_{\ell=1}^\infty a_\ell \frac{\sqrt{2} \sin ((\ell-1/2) \pi x)}{(\ell-1/2)\pi},
\end{equation}
i.e., every continuous function $f$ which vanishes at $0$ can be described by Eq. \eqref{exp_rand_fun}. We can identify a function $f \in \cH$ with the sequence of coefficients $\bsa = (a_\ell)_{\ell \in \mathbb{N}}$ via the injective mapping $T: \cH \to \mathbb{R}^{\mathbb{N}}$, where $T(f) = \bsa$. To define a probability measure on $\cH$, it thus suffices to define a probability measure on the set of sequences $\bsa$.

In one dimension, we use the Gaussian distribution with mean $0$ and variance $1$ and for sequences we use the infinite product measure. That is, the measure of any interval $[\bsb, \bsc] := \prod_{j=1}^\infty [b_j, c_j]$, with $b_j \le c_j$, is given by
\begin{equation*}
\mathbb{Q}([\bsb,\bsc]) := \prod_{j=1}^\infty \frac{1}{\sqrt{2\pi}} \int_{b_j}^{c_j} \exp\left(-\frac{x^2}{2}\right) \,\mathrm{d} x.
\end{equation*}
In other words, the probability that $\bsa \in [\bsb, \bsc]$ is given by $\mathbb{Q}([\bsb,\bsc])$. This can then be extended to any Borel set $A \in \mathbb{R}^{\mathbb{N}}$. The Borel $\sigma$-algebra on $\mathbb{R}^{\mathbb{N}}$ defines a $\sigma$-algebra $\mathcal{F}$ on $\cH$ via the mapping $T$. For a Borel set $A \subseteq \mathbb{R}^{\mathbb{N}}$ let $F_A = \{f \in \cH: T(f) \in A\}$. The probability measure $\mathbb{P}$ on $(\cH, \mathcal{F})$ is now given by
\begin{equation*}
\mathbb{P}(F_A) = \mathbb{Q}(A) \quad \mbox{for any Borel set } A.
\end{equation*}

It is known that if one chooses the coefficients $a_\ell$ in \eqref{exp_rand_fun} i.i.d. with Gaussian distribution with mean $0$ and variance $1$, then the function $f$ is almost surely continuous. This follows since a Wiener process or Brownian motion is almost surely continuous. This means that
\begin{equation*}
\mathbb{Q}(\mathbb{R}^{\mathbb{N}} \setminus T(\cH) ) = 0.
\end{equation*}

The covariance kernel is now given by
\begin{align*}
C(x,y) = & \int_{\cH} f(\bsx) f(\bsy) \, \mathbb{P}(\mathrm{d} f) \\ = & \sum_{k, \ell=1}^\infty  \EE(a_k a_\ell) \frac{\sqrt{2} \sin ( (k-1/2) \pi x)}{(k-1/2) \pi} \frac{\sqrt{2} \sin( (\ell-1/2) \pi y)}{(\ell-1/2) \pi}.
\end{align*}
The expectation value for $k \neq \ell$ is $0$, whereas for $k = \ell$ it is $1$, since the mean of $a_k$ is $0$ and the variance is $1$. Thus
\begin{align}\label{defC}
C(x, y) = &  \sum_{\ell=1}^\infty \frac{\sqrt{2} \sin ( (\ell-1/2) \pi x)}{(\ell-1/2) \pi} \frac{\sqrt{2} \sin( (\ell-1/2) \pi y)}{(\ell-1/2) \pi}  =  \min(x, y).
\end{align}

\subsubsection*{Average-case error}

We have seen how reproducing kernels can be used to give a formula for the worst-case error. We now provide an analogue for the covariance kernel and the average-case error.

Let $\cH$ be a function space defined on $[0,1]^s$ and let $(\cH, \mathcal{F}, \mathbb{P})$ be a probability space. For $1 \le p \le \infty$ we define the {\it $L_p$ average-case} error by
\begin{equation*}
\mathrm{ace}_p(\cH, \cP) = \left( \int_{\cH} \left|\int_{[0,1]^s} f(\bsx) \,\mathrm{d} \bsx - \frac{1}{N} \sum_{\bsx \in \cP} f(\bsx) \right|^p \, \mathbb{P}(\mathrm{d} f) \right)^{1/p},
\end{equation*}
with the obvious modifications for $p=\infty$.

We consider now the case $p =2$. Let $C:[0,1]^s \times [0,1]^s \to \mathbb{R}$ be the covariance kernel, that is
\begin{equation*}
C(\bsx, \bsy) = \int_{\cH} f(\bsx) f(\bsy) \mathbb{P}(\mathrm{d} f).
\end{equation*}
Then we have
\begin{align*}
\mathrm{ace}^2_2(\cH, \cP) = & \int_{\cH} \int_{[0,1]^s} \int_{[0,1]^s} f(\bsx) f(\bsy) \,\mathrm{d} \bsx \,\mathrm{d} \bsy \,\mathbb{P}(\mathrm{d} f) \\ & - \int_{\cH}  \frac{2}{N} \sum_{\bsx \in \cP} \int_{[0,1]^s} f(\bsx) f(\bsy) \,\mathrm{d} \bsy \,\mathbb{P}(\mathrm{d} f)  + \int_{\cH} \frac{1}{N^2} \sum_{\bsx, \bsy \in \cP} f(\bsx) f(\bsy) \,\mathbb{P}(\mathrm{d} f) \\  = &  \int_{[0,1]^s} \int_{[0,1]^s}  \int_{\cH} f(\bsx) f(\bsy) \, \mathbb{P}(\mathrm{d} f) \,\mathrm{d} \bsx \,\mathrm{d} \bsy \\ & - \frac{2}{N} \sum_{\bsx \in \cP} \int_{[0,1]^s} \int_{\cH} f(\bsx) f(\bsy) \,\mathbb{P}(\mathrm{d} f) \,\mathrm{d} \bsy  + \frac{1}{N^2} \sum_{\bsx, \bsy \in \cP} \int_{\cH} f(\bsx) f(\bsy) \,\mathbb{P}(\mathrm{d} f) \\  = & \int_{[0,1]^s} \int_{[0,1]^s} C(\bsx,\bsy) \,\mathrm{d} \bsx \,\mathrm{d} \bsy - \frac{2}{N} \sum_{\bsx \in \cP} \int_{[0,1]^s} C(\bsx, \bsy) \,\mathrm{d} \bsy + \frac{1}{N^2} \sum_{\bsx, \bsy \in \cP} C(\bsx, \bsy).
\end{align*}
This formula is analogous to \eqref{wce_k}. Since symmetric positive definite functions can be interpreted as reproducing kernels or covariance kernels, this allows one to interpret the error either as worst-case error or as average-case error (for a different function space).  We refer the reader to \cite[Chapter~24]{NW12} and \cite{ritt} for more information on covariance kernels and average-case errors.




\subsection{Karhunen-Lo\'eve Expansion}

The Karhunen-Lo\'eve expansion of the covariance kernel follows from Mercer's theorem by using the fact that the covariance kernel is also a reproducing kernel.

\begin{theorem}
Let $Z(\bsx)$ be a zero-mean square integrable random field (stochastic process) defined over a probability space $(\Omega, \mathcal{F}, \mathbb{P})$ and indexed over the interval $[0,1]^s$, with continuous covariance kernel $C: [0,1]^s \times [0,1]^s \to \mathbb{R}$. Then $C$ satisfies the conditions in Mercer's theorem and we have the expansion
\begin{equation*}
C(\bsx, \bsy) = \sum_{\ell=1}^\infty \lambda_\ell \psi_\ell(\bsx) \overline{ \psi_\ell(\bsy) },
\end{equation*}
where $\psi_\ell$ are $L_2([0,1]^s)$ orthonormal eigenfunctions with corresponding eigenvalues $(\lambda_\ell)_{\ell=1}^{\infty}$. Then the random field (stochastic process)  $Z(\bsx)$ admits the presentation
\begin{equation*}
Z(\bsx) = \sum_{\ell=1}^\infty \xi_\ell \sqrt{\lambda_\ell } \psi_{\ell}(\bsx),
\end{equation*}
where the convergence is in $L_2$ norm, uniform in $\bsx$ and
\begin{equation*}
\xi_\ell = \frac{1}{\sqrt{\lambda_\ell }} \int_{[0,1]^s} Z(\bsx) \overline{ \psi_\ell(\bsx) } \,\mathrm{d} \bsx.
\end{equation*}
The random variables $\xi_\ell$ have zero-mean, are uncorrelated and have variance $1$.
\end{theorem}

The Karhnunen-Lo\'eve expansion yields a bi-orthogonal expansion of a random field (stochastic process), since the random variables $\xi_\ell$ are uncorrelated and hence $\mathbb{E}(\xi_\ell \xi_k) = \delta_{\ell, k}$, the Kronecker $\delta$ symbol, and the eigenfunctions are $L_2$ orthonormal.

The Wiener process or Brownian motion can be expanded in terms of its Karhunen-Lo\'eve expansion, which we describe in the following.

\subsubsection*{Example: Karhunen-Lo\'eve expansion of Wiener process or Brownian motion}

The covariance kernel of the {\it Wiener process} is given by
\begin{equation*}
W(x,y) = \min(x, y).
\end{equation*}
We have analyzed the corresponding reproducing kernel in Section~\ref{subsec_Mercer}. Functions in the corresponding reproducing kernel Hilbert space have the expansion given in \eqref{fi_Wiener}.

We can now use this expansion to describe a Wiener process (or also called Brownian motion) on the interval $[0,1]$. Compared to its deterministic counterpart (i.e. functions in the corresponding reproducing kernel Hilbert space), the coefficients in the expansion are now random variables.

Let $\xi_\ell \in \mathcal{N}(0,1)$ for $\ell \in \mathbb{N}$ be independent Gaussian random variables with mean $0$ and variance $1$. Then the Wiener process $Z(x)$ has the expansion
\begin{equation*}
Z(x) = \sum_{\ell=1}^\infty \xi_\ell \frac{\sqrt{2} \sin ((\ell-1/2) \pi x)}{(\ell-1/2) \pi}.
\end{equation*}
It is easy to see that the expectation value of $Z(x)$ satisfies $\mathbb{E}(Z(x)) = 0$, since all $\xi_\ell$ have mean $0$. The covariance is now given by
\begin{align*}
\mathrm{cov}(Z(x), Z(y)) = & \mathbb{E}\left(\sum_{\ell=1}^\infty \xi_\ell \frac{\sqrt{2}\sin ((\ell-1/2) \pi x)}{(\ell-1/2) \pi} \sum_{k=1}^\infty \xi_k \frac{\sqrt{2} \sin ((k-1/2) \pi y)}{(k-1/2) \pi } \right) \\ = & \sum_{k, \ell=1}^\infty \mathbb{E}(\xi_\ell \xi_k) \frac{\sqrt{2}\sin ((\ell-1/2) \pi x)}{(\ell-1/2) \pi} \frac{\sqrt{2}\sin ((k -1/2) \pi y)}{(k -1/2) \pi}.
\end{align*}
Since the random variables $\xi_\ell$ are independent with mean $0$ we have $\mathbb{E}(\xi_\ell \xi_k) = 0$ for $k \neq \ell$. If $k = \ell$ it follows that $\mathbb{E}(\xi_\ell \xi_\ell) = 1$, since the variance of $\xi_\ell$ is also $1$. Thus we have
\begin{equation*}
\mathrm{cov}(Z(x), Z(y)) =  \sum_{\ell=1}^\infty \frac{\sqrt{2}\sin ((\ell-1/2) \pi x)}{(\ell-1/2) \pi} \frac{\sqrt{2}\sin ((\ell -1/2) \pi y)}{(k -1/2) \pi} = \min(x, y)=C(x,y),
\end{equation*}
where $C$ is as in \eqref{defC}.

A smooth version of the Brownian motion can be obtained via integration. The covariance kernel of the integrated Brownian motion is discussed in \cite{GHT03}.

\subsubsection*{Partial differential equations with random coefficients}

As an application of random fields we describe partial differential equations (PDE) with random coefficients.

We consider the physical domain $[0,1]^d$ (usually $d=1,2,3$). Let
\begin{equation*}
a(\bsx, \bsz) = a_0(\bsx) + \sum_{\ell=1}^\infty z_\ell \lambda_\ell \psi_\ell(\bsx),
\end{equation*}
where $\bsz = (z_1, z_2, \ldots)$. The $z_\ell$ are i.i.d. random variables with mean $0$ and finite variance $\sigma$. In the simplest case, the $z_\ell$ are uniformly distributed in $[-1/2, 1/2]$, but other distributions can be studied as well. Then $a-a_0$ is a random field with mean $0$, or, in other words, the mean of $a$ is $a_0$. The underlying covariance kernel $C$ corresponding to $a-a_0$ is given by
\begin{equation*}
C(\bsx, \bsy) = \sum_{\ell=1}^\infty \sigma \lambda_\ell^2 \psi_\ell(\bsx) \overline{\psi_\ell(\bsy)}={\rm cov}(a(\bsx,\cdot),a(\bsy,\cdot)).
\end{equation*}

We consider now the PDE
\begin{equation*}
- \nabla \cdot (a(\bsx, \bsz) \nabla u(\bsx, \bsz) ) = f(\bsx) \mbox{ in } D = [0,1]^d, \quad u(\bsx, \bsz) = 0 \mbox{ on } \partial D.
\end{equation*}

Since the $z_\ell$ are random variables, the solution $u$ of the PDE also depends on the random variables $z_\ell$, and is therefore also a random variable. One is for instance interested in approximating the expectation value of $u$ (or a linear functional of $u$). To approximate the expectation value of the solution $u$, one ansatz is to use QMC points to sample $(z_1, z_2, \ldots, z_s)$ for some large enough $s$, set $z_{s+1} = z_{s+2} = \ldots = 0$ and use a PDE solver to approximate $u(\bsx, (z_1, z_2, \ldots, z_s, 0, 0, \ldots))$. Averaging the solution $u$ over all QMC points yields an approximation of the expectation value. Such a study is carried out in \cite{KSS1}. See also \cite{GKNSS} where the covariance kernel was used directly to sample from the random field.

\subsection{Lower Bounds Using Bump Functions}\label{subsec_bumbf}

A standard approach to proving lower bounds involves so-called {\it bump functions}. Let $\cH$ be a Banach space with norm $\|\cdot\|$. To prove a lower bound on the worst-case error one possible strategy is to construct a bump function. Let $\cP = \{\bsx_0, \bsx_1, \ldots, \bsx_{N-1}\} \subseteq [0,1]^s$ be an arbitrary but fixed point set. The idea is to construct a function $f$ with the following properties:
\begin{enumerate}
\item $f(\bsx_n) = 0$ for all $0 \le n < N$;
\item $\|f\| = 1$;
\item $\int_{[0,1]^s} f(\bsx) \,\mathrm{d} \bsx$ is large.
\end{enumerate}
If we can construct for every $N$-point set $\cP$ such a function $f$ which satisfies those three properties, with $\int_{[0,1]^s} f(\bsx) \,\mathrm{d} \bsx \ge \varepsilon(N, s)$, say, then
\begin{equation*}
\inf_{\cP \subseteq [0,1]^s \atop |\cP| = N} \mathrm{wce}(\cH, \cP) \ge \varepsilon(N, s).
\end{equation*}

We illustrate the idea in a simple example. 
\begin{theorem}\label{lower_bound_bump}
Let $\cH_K$ be the reproducing kernel Hilbert space with reproducing kernel $K(\bsx, \bsy) = \prod_{j=1}^s ( 1 + \min(x_j, y_j))$.  Let $\cP = \{\bsx_0, \bsx_1, \ldots, \bsx_{N-1}\} \subseteq [0,1]^s$ be an arbitrary point set. Then there exists a constant $c_s > 0$ independent of $N$ and $\cP$ such that
\begin{equation*}
\mathrm{wce}(\cH_K, \cP) \ge c_s \frac{(\log N)^{\frac{s-1}{2}}}{N}.
\end{equation*}
\end{theorem}

\begin{proof}
To construct $f$, we start with the one-dimensional case. One choice of a basic function $\phi:\mathbb{R} \to \mathbb{R}$ is
\begin{equation*}
\phi(t) = \left\{ \begin{array}{ll}  t (1-t) & \mbox{if } 0 < t < 1, \\ 0 & \mbox{otherwise}. \end{array} \right.
\end{equation*}
(If one considers function spaces of smoothness $r$, then one could use $t^{r}(1-t)^{r}$.) The scaled and shifted versions are
\begin{equation*}
\phi(2^m t - a)
\end{equation*}
for integers $m \in \mathbb{N}_0$ and $0 \le a < 2^m$. The support of this scaled and shifted function is $[a/2^m, (a+1)/2^m]$.

Choose the integer $m$ such that $2^{m-2} \le N < 2^{m-1}$. Let $\bsm = (m_1, m_2,\ldots, m_s) \in \mathbb{N}_0$ and let $|\bsm|=m_1+m_2+\cdots+m_s$. Define $\DD_j=\{0,1,\ldots,2^j-1\}$ and $\DD_{\bsm}=\DD_{m_1} \times \ldots \times \DD_{m_s}$. We can now define a function $g_{\bsm}$ which satisfies 1. by setting
\begin{equation*}
g_{\bsm}(\bsx) = \sum_{\satop{ \bsa \in \DD_{\bsm} }{(\bsa/2^{\bsm}, (\bsa+\boldsymbol{1})/2^{\bsm}) \cap \cP = \emptyset} }  \prod_{j=1}^s \phi(2^{m_j} x_j - a_j),
\end{equation*}
where $(\bsa/2^{\bsm}, (\bsa+\boldsymbol{1})/2^{\bsm} ) = \prod_{j=1}^s (a_j/2^{m_j}, (a_j+1)/2^{m_j} )$. The condition $(\bsa/2^{\bsm}, (\bsa+\boldsymbol{1})/2^{\bsm}) \cap \cP = \emptyset$ ensures that $g_{\bsm}(\bsx_n) = 0$ for all $0 \le n < N$. We define the function
\begin{equation}\label{bump_fun_g}
g(\bsx ) = \sum_{\satop{\bsm \in \mathbb{N}_0^s}{|\bsm| = m}} g_{\bsm}(\bsx).
\end{equation}
Again we have $g(\bsx_n) = 0$ for all $0 \le n < N$.

In the next step, we estimate the norm of $g$ and then set $f = g/\|g\|_K$. Then $f$ also satisfies the second condition. The squared norm in our particular function space is given by
\begin{align*}
\|h\|_K^2 = & \sum_{\uu \subseteq [s]} \int_{[0,1]^{\uu}} \left| \frac{\partial^{\uu} h}{\partial \bsx_{\uu}}(\bsx_{\uu}; \boldsymbol{0}) \right|^2 \,\mathrm{d} \bsx_{\uu},
\end{align*}
where $[s]=\{1,2,\ldots,s\}$, and for $\uu \subseteq [s]$ and $\bsx=(x_1,x_2,\ldots,x_s)$ we write $(\bsx_{\uu}; \boldsymbol{0})$ for the $s$-dimensional vector whose $j$th component is $x_j$ for $j \in \uu$ and $0$ otherwise.

We consider now the norm of \eqref{bump_fun_g}. Since for $t=0$ we have $\phi(2^m t- a) = \phi(-a) = 0$ for all integers $a$, we obtain that
\begin{align*}
\|g\|_K^2 = & \int_{[0,1]^s} \left|\frac{\partial^s g}{\partial \bsx}(\bsx) \right|^2 \,\mathrm{d} \bsx \\ = &  \sum_{\satop{\bsm \in \mathbb{N}_0^s}{|\bsm| = m}}  \sum_{\satop{\bsm' \in \mathbb{N}_0^s}{|\bsm'| = m}} \int_{[0,1]^s} \frac{\partial^s g_{\bsm} }{\partial \bsx}(\bsx) \frac{\partial^s g_{\bsm'} }{\partial \bsx}(\bsx) \,\mathrm{d} \bsx \\ = & \sum_{\satop{\bsm \in \mathbb{N}_0^s}{|\bsm| = m}} \sum_{\satop{\bsm' \in \mathbb{N}_0^s}{|\bsm'| = m}} \sum_{\satop{ \bsa \in \DD_{\bsm} }{(\bsa/2^{\bsm}, (\bsa+\boldsymbol{1})/2^{\bsm} ) \cap \cP = \emptyset} }  \sum_{\satop{ \bsa' \in \DD_{\bsm'}}{(\bsa'/2^{\bsm'}, (\bsa'+\boldsymbol{1})/2^{\bsm'}) \cap \cP = \emptyset} } \prod_{j=1}^s  \int_0^1   \phi'(2^{m_j} x_j - a_j)  \phi'(2^{m'_j} x_j - a'_j) \,\mathrm{d} x_j.
\end{align*}
For $m_j \ge m'_j$ we have
\begin{equation*}
\int_0^1 \phi'(2^{m_j} x_j - a_j) \phi'(2^{m'_j} x_j - a'_j) \,\mathrm{d} x_j = \left\{\begin{array}{ll}  \frac{1}{3} 2^{2m'_j - m_j} & \mbox{if } \left[\frac{a_j}{2^{m_j}}, \frac{a_j+1}{2^{m_j}}\right] \subseteq \left[\frac{a'_j}{2^{m'_j}},\frac{a'_j+1}{2^{m'_j}}\right], \\ 0 & \mbox{otherwise}. \end{array} \right.
\end{equation*}
Note that $[a_j 2^{-m_j}, (a_j+1) 2^{-m_j}]$ is the support of $\phi'(2^{m_j} x_j - a_j)$. The condition that the support of $\phi'(2^{m_j} x_j - a_j)$ is contained in the support of $\phi'(2^{m'_j} x_j - a'_j)$ is equivalent to $2^{m_j - m'_j} a'_j \le a_j < 2^{m_j - m'_j} (a'_j + 1)$. Thus for given $a'_j, m_j, m'_j$ there are $2^{m_j-m'_j}$ possible choices for $a_j$. Thus we have
\begin{align*}
\|g\|_K^2 \le & \frac{1}{3^s} \sum_{\satop{\bsm \in \mathbb{N}_0^s}{|\bsm|= m}}  \sum_{\satop{\bsm' \in \mathbb{N}_0^s}{|\bsm'|= m}} \prod_{j=1}^s \left( 2^{2\min\{m_j, m'_j\} - \max\{m_j, m'_j\}} 2^{\min\{m_j, m'_j\} } 2^{\max\{m_j, m'_j\} - \min\{m_j, m'_j\}} \right) \\ = & \frac{2^{2m}}{3^s} \sum_{\satop{\bsm \in \mathbb{N}_0^s}{|\bsm|= m}}  \sum_{\satop{\bsm' \in \mathbb{N}_0^s}{|\bsm'| = m}} \prod_{j=1}^s 2^{-|m_j-m'_j| }.
\end{align*}
For any fixed $\bsm \in \mathbb{N}_0^s$ we have
\begin{align*}
\sum_{\satop{\bsm' \in \mathbb{N}_0^s}{|\bsm'|= m}} \prod_{j=1}^s 2^{-|m_j-m'_j| } \le & \sum_{\satop{\bsk \in \mathbb{Z}^s}{k_1+ k_2 + \cdots + k_s = 0}} 2^{-|k_1| - |k_2| - \cdots - |k_s|} \le  \left(\sum_{k=-\infty}^\infty 2^{-|k| } \right)^s  =  3^{s}.
\end{align*}
This implies that
\begin{equation*}
\|g\|_K^2 \le 2^{2m}   \sum_{\satop{\bsm \in \mathbb{N}_0^s}{|\bsm|= m}} 1 \le 2^{2m} {m+s-1 \choose s-1}.
\end{equation*}

Further we have
\begin{align*}
\int_{[0,1]^s} g(\bsx) \,\mathrm{d}\bsx = & \sum_{\satop{\bsm \in \mathbb{N}_0^s}{|\bsm| = m}} \sum_{\satop{ \bsa \in \DD_{\bsm}}{(\bsa/2^{\bsm}, (\bsa+\boldsymbol{1})/2^{\bsm}) \cap \cP = \emptyset} }  \prod_{j=1}^s \int_0^1 \phi(2^{m_j} x_j - a_j) \,\mathrm{d} x_j \\ = &\sum_{\satop{\bsm \in \mathbb{N}_0^s}{|\bsm|= m}} \sum_{\satop{ \bsa \in \DD_{\bsm} }{(\bsa/2^{\bsm}, (\bsa+\boldsymbol{1})/2^{\bsm}) \cap \cP = \emptyset} }  \frac{1}{2^m 6^s} \\ \ge & {m +s-1 \choose s-1} \frac{2^m-N}{2^m 6^s} \ge {m+s-1 \choose s-1} \frac{1}{2\cdot 6^s}.
\end{align*}

Let now $f = g / \|g\|_K$. Then we have $\|f\|_K = 1$ and there is a constant $c_s > 0$ such that
\begin{align*}
\int_{[0,1]^s} f(\bsx) \,\mathrm{d} \bsx = & \frac{1}{\|g\|} \int_{[0,1]^s} g(\bsx) \,\mathrm{d}\bsx \ge \frac{1}{2 \cdot 6^s} \frac{1}{2^m} \sqrt{{m+s-1 \choose s-1}} \ge  c_s \frac{(\log N)^{\frac{s-1}{2}}}{N}.
\end{align*}

Since $f$ satisfies all three conditions, we obtain
\begin{equation*}
\mathrm{wce}(\cH, \cP) \ge c_s \frac{(\log N)^{\frac{s-1}{2}}}{N}
\end{equation*}
for any $N$-element point set $\cP \subseteq [0,1]^s$.
\end{proof}

%

\subsection{The Rader Transform}

The {\it Rader transform} can be used to permute certain matrices such that the resulting matrices are circulant. Circulant matrices are very useful since a fast matrix-vector multiplication using the fast Fourier transform exists in this case. The Rader transform is used in the fast component-by-component construction of lattice rules (see Section~\ref{secLPS}) and polynomial lattice rules (see Section~\ref{secPafLs}). The Rader transform in the context of the component-by-component construction was introduced in \cite{NCa, NCb, NCc}.

We explain a special case of the Rader transform in the context of lattice rules. Let $N$ be a prime number and let $\omega: \{0, 1, \ldots, N-1\} \to \mathbb{R}$ be an arbitrary mapping. Let $C = (c_{k, \ell})_{1 \le k, \ell < N}$ be the $(N-1) \times (N-1)$ matrix with
\begin{equation*}
c_{k, \ell} = \omega( k \ell \pmod{N}).
\end{equation*}
In the following we show how the Rader transform can be used to obtain permutation matrices $P$ and $Q$ such that $P C Q$ is a circulant matrix. A matrix $D = (d_{k, \ell})$ is {\it circulant} if $d_{k, \ell} = e_{k - \ell \pmod{N-1}}$ for some numbers $e_0, e_1, \ldots, e_{N-2} \in \mathbb{R}$.

Let $\mathbb{F}_N = \{0, 1, \ldots, N-1\}$ be the finite field of order $N$ (we identify the elements in $\mathbb{Z}_N$ with the integers $0, 1, \ldots, N-1$). Then there exists a primitive element $g \in \mathbb{F}_N$, that is, the multiplicative group $\mathbb{F}_N^{\times}$ of $\mathbb{F}_N$ is given by
\begin{equation*}
\mathbb{F}_N^{\times} = \{g^0, g^1, g^2, \ldots, g^{N-2}\}.
\end{equation*}
Note that we always have $g^{N-1} = 1$. Let $D= (d_{k,\ell})$ where
\begin{equation*}
d_{k, \ell} = e_{k-\ell \pmod{N-1}} = \omega(g^{k-\ell} \pmod{N} ).
\end{equation*}

We define now the permutation matrix $\Pi(g) = (\pi_{k,\ell}(g))_{1 \le k, \ell < N}$ by
\begin{equation*}
\pi_{k, \ell}(g) = \left\{\begin{array}{rl} 1 & \mbox{if } \ell = g^k \pmod{N}, \\ 0 & \mbox{otherwise}. \end{array} \right.
\end{equation*}
Then we have
\begin{equation*}
D = \Pi(g) C \Pi(g^{-1})^\top
\end{equation*}
and the matrix $D$ is a circulant matrix, since
\begin{equation*}
d_{k, \ell} = \sum_{u, v=1}^{N-1} \pi_{k, u}(g) c_{u, v} \pi_{\ell, v}(g^{-1}) = c_{g^k, g^{-\ell}} = \omega(g^{k-\ell} \pmod{N}).
\end{equation*}


\section{Harmonic Analysis}\label{sec_har_anal}

Methods from harmonic analysis used in QMC range from basic applications of orthogonality like Parseval's equality and Bessel's inequality to sophisticated tools like Riesz products and Littlewood-Paley theory. In this section we explain some of the tools by showing some central results in simplified settings.

\subsection{Orthogonal Bases - Error Bounds for QMC}\label{OBEB}

Orthogonal bases in $L_2\big([0,1]^s\big)$ useful for the analysis of errors of QMC rules and discrepancy estimates are
\begin{itemize}
	\item the trigonometric bases
	\item Walsh bases
	\item Haar bases.
\end{itemize}   
The first two are systems of characters on $[0,1]^s$ with respect to different group structures which makes them very suitable for the analysis of point sets respecting that group structure (see Section~\ref{grchardual}). The Haar bases have the advantage that the orthogonal functions are local and can be used to characterize function spaces through wavelet decompositions.   

The {\it trigonometric system} contains the {\it trigonometric functions} defined by ${\rm e}_{\bsk}: [0,1)^s \rightarrow \CC$ for $\bsk \in \ZZ^s$ by
$$
{\rm e}_{\bsk}(\bsx)=\exp(2 \pi \icomp \bsk \cdot \bsx) \ \ \mbox{ for } \bsx \in [0,1)^s,
$$
where ``$\cdot$'' denotes the usual inner product in $\RR^s$. The trigonometric system is an orthonormal basis of the Hilbert space $L_2\big([0,1]^s\big)$ whose inner product we denote with 
$\langle \cdot , \cdot \rangle$.
One main application of the trigonometric system in QMC is the error analysis of lattice rules.
Lattices and lattice rules are discussed in more detail in Section~\ref{secLPS}.
Here we consider for simplicity just rank-1 lattice rules, which are of the form 
$$ \cP(\bsg,N)=\left\{ \left\{\frac{n}{N} \bsg\right\}\ : \  n=0,1,\ldots,N-1 \right\}$$
for some $N \in \NN$, $N \ge 2$ and some generator $\bsg \in \ZZ^s$, where the fractional part function $\{\cdot\}$ is applied component-wise.

\subsubsection*{Example: Error analysis of rank-1 lattice rules}

Let $f: \RR^s \rightarrow \CC$ be a 1-periodic function (in each variable) with absolutely convergent Fourier series
$$ f = \sum_{\bsk\in\ZZ^s} \widehat{f} (\bsk) {\rm e}_{\bsk} $$
with the Fourier coefficients $\widehat{f}=\langle f ,  {\rm e}_{\bsk} \rangle$. By periodicity, the rank-1 lattice rule with generator $\bsg$ can be written as
$$ \int_{[0,1]^s} f(\bsx) \dint \bsx \approx \frac{1}{N} \sum_{n=0}^{N-1} f \left( \frac{n\bsg}{N} \right). $$ 
Since the integral is just $\widehat{f} (0)$, we get for the error
\begin {eqnarray*}
  \frac{1}{N} \sum_{n=0}^{N-1} f \left( \frac{n\bsg}{N} \right) - \int_{[0,1]^s} f(\bsx) \dint \bsx 
	&=&
	\frac{1}{N} \sum_{n=0}^{N-1} \sum_{\bsk\in\ZZ^s} \widehat{f} (\bsk) {\rm e}_{\bsk} \left( \frac{n\bsg}{N} \right) - \widehat{f} (0) \\
	&=&
	\sum_{\bsk\in\ZZ^s} \widehat{f} (\bsk) \frac{1}{N} \sum_{n=0}^{N-1}  {\rm e}_{\bsk} \left( \frac{n\bsg}{N} \right) - \widehat{f} (0) \\
	&=&
	\sum_{\bsk\in\ZZ^s \setminus \{0\}} \widehat{f} (\bsk) \frac{1}{N} \sum_{n=0}^{N-1} {\rm e}_{\bsk} \left( \frac{n\bsg}{N} \right).
\end{eqnarray*}
Now a character property of the trigonometric functions (see Lemma~\ref{explatchar} in Section~\ref{grchardual}) implies that the inner sum is 1 if $\bsk \cdot \bsg \equiv 0 \pmod{N}$ and 0 otherwise. Hence
$$ \frac{1}{N} \sum_{n=0}^{N-1} f \left( \frac{n\bsg}{N} \right) - \int_{[0,1]^s} f(\bsx) \dint \bsx = \sum_{\bsk}\widehat{f} (\bsk), $$
where the last sum runs only over those $\bsk \neq \boldsymbol{0}$ with $\bsk \cdot \bsg \equiv 0 \pmod{N}$.
This condition defines the dual lattice (cf. Section~\ref{grchardual}), and the error characterization can be extended accordingly to general lattices.
Smoothness conditions on $f$ can be encoded in decay conditions for the Fourier coefficients. So, to get a small error for the integration of smooth functions, the lattice generator should be chosen such that the dual lattice avoids the Fourier coefficients with large $\bsk$. For more information we refer to \cite{niesiam,slojoe}. \\

As the trigonometric system is well adapted to study lattice  rules, Walsh bases can be similarly used for digital constructions, see Section \ref{secDig}.\\

We now turn to the Haar system. We restrict to the base 2 case, applications of Haar bases in base $b\ge 2$ can be found in \cite{Mar2013,Mar2013a,Mar2013b}.
A {\it dyadic interval} of length $2^{-j}, j\in {\mathbb N}_0,$ in $[0,1)$ is an interval of the form 
$$ I=I_{j,m}:=\left[\frac{m}{2^j},\frac{m+1}{2^j}\right) \ \ \mbox{for } \  m=0,1,\ldots,2^j-1.$$ 
The left and right half of $I=I_{j,m}$ are the dyadic intervals $I^+ = I_{j,m}^+ =I_{j+1,2m}$ and $I^- = I_{j,m}^- =I_{j+1,2m+1}$, respectively. The {\it Haar function} $h_I = h_{j,m}$ with support $I$ 
is the function on $[0,1)$ which is  $+1$ on the left half of $I$, $-1$ on the right half of $I$ and 0 outside of $I$. The $L_\infty$-normalized {\it Haar system} consists of
all Haar functions $h_{j,m}$ with $j\in{\mathbb N}_0$ and  $m=0,1,\ldots,2^j-1$ together with the indicator function $h_{-1,0}$ of $[0,1)$.
Normalized in $L_2([0,1))$ we obtain the {\it orthonormal Haar basis} of $L_2([0,1))$. 

Let ${\mathbb N}_{-1}=\{-1,0,1,2,\ldots\}$ and define ${\mathbb D}_j=\{0,1,\ldots,2^j-1\}$ for $j\in{\mathbb N}_0$ and ${\mathbb D}_{-1}=\{0\}$ for $j=-1$.
For $\bsj=(j_1,\dots,j_s)\in{\mathbb N}_{-1}^s$ and $\bsm=(m_1,\dots,m_s)\in {\mathbb D}_{\bsj} :={\mathbb D}_{j_1}\times \ldots \times {\mathbb D}_{j_s}$, 
the {\it Haar function} $h_{\bsj,\bsm}$
is given as the tensor product 
$$h_{\bsj,\bsm} (x) = h_{j_1,m_1}(x_1)\, \cdots \, h_{j_s,m_s}(x_s) \ \ \ \mbox{ for } \bsx=(x_1,\dots,x_s)\in[0,1)^s.$$
The boxes $$I_{\bsj,\bsm} = I_{j_1,m_1} \times \ldots \times I_{j_s,m_s}$$ are called {\it dyadic boxes}.
Two boxes $I_{\bsj_1,\bsm_1}$ and $I_{\bsj_2,\bsm_2}$ have the {\it same shape} if $\bsj_1=\bsj_2$.
A crucial combinatorial property is that for $\bsj=(j_1,\dots,j_s)\in{\mathbb N}_{0}^s$, there are exactly $2^{j_1+\dots+j_s}$ boxes of that shape which are mutually disjoint.
If we fix the level $\ell=j_1+\dots+j_s$, then there are 
$$ \binom{\ell+s-2}{s-1} \approx_s \ell^{s-1} $$
different shapes of boxes with level $\ell$.

The $L_\infty$-normalized tensor {\it Haar system} consists of all Haar functions $h_{\bsj,\bsm}$ with $\bsj\in{\mathbb N}_{-1}^s$ and  
$\bsm \in {\mathbb D}_j$. Normalized in $L_2([0,1)^s)$ we obtain the {\it orthonormal Haar basis} of $L_2([0,1)^s)$.  

\subsubsection*{Example: Error analysis of QMC with Hammersley point sets}

The Haar coefficients can be used directly to compute and estimate the norm of the discrepancy function.
As an example, we compute the $L_2$-discrepancy (see Section~\ref{subsec_KH_ugl}) of the {\it two-dimensional symmetrized Hammersley type point set} given by
$$
 {\cal R}_n = \Big\{ \Big( \frac{t_n}{2}+\frac{t_{n-1}}{2^2}+\cdots + \frac{t_1}{2^n} ,  
                       \frac{s_{1}}{2}+\frac{s_{2}}{2^2}  +\cdots + \frac{s_n}{2^n} \Big) \ : \ 
                       t_1,\ldots,t_n \in \{0,1\}   \Big\} 
$$
where $s_i=t_i$ if $i$ is even and $s_i=1-t_i$ if $i$ is odd. 
The cardinality of this set is $N=2^n$.
It was shown in \cite{HZ1969} that these sets satisfy the $L_2$ discrepancy estimate
$$ D_N({\cal R}_n, L_2) \ll  \frac{\sqrt{\log N}}{N},$$
which is optimal according to Theorem~\ref{thm:roth} in the next section. An exact formula for $D_N({\cal R}_n, L_2)$ and a generalization of the result can be found in \cite{KP06}.

Direct, but in some cases a little tedious computations, for which we refer to \cite{hin2010}, give the Haar coefficients $\mu_{j,m} = \langle  D_N({\cal R}_n, \, \cdot \,), h_{\bsj,\bsm} \rangle$ as follows:
\begin{lemma}
 \label{lem:HCH}
 Let $\bsj=(j_1,j_2)\in \NN_{0}^2$. 
 Then
 \begin{itemize}
  \item[(i)] if $j_1+j_2<n-1$ and $j_1,j_2\ge 0$ then $|\mu_{\bsj,\bsm}| = 2^{-2(n+1)}$.
  \item[(ii)] if $j_1+j_2\ge n-1$ and $0\le j_1,j_2\le n$ then $|\mu_{\bsj,\bsm}| \le 2^{-(n+j_1+j_2+1)}$ and
     $|\mu_{\bsj,\bsm}| = 2^{-2(j_1+j_2+2)}$ for all but at most $2^n$ coefficients $\mu_{\bsj,\bsm}$ with $\bsm\in {\mathbb D}_{\bsj}$.
  \item[(iii)] if $j_1 \ge n$ or $j_2 \ge n$ then $|\mu_{\bsj,\bsm}| = 2^{-2(j_1+j_2+2)}$.
 \end{itemize} 
 
 Now let $\bsj=(-1,k)$ or $\bsj=(k,-1)$ with $k\in \NN_0$. Then
 \begin{itemize}
  \item[(iv)] if $k<n$ then $|\mu_{\bsj,\bsm}| \le 2^{-(n+k)}$.
  \item[(v)] if $k\ge n$  then $|\mu_{\bsj,\bsm}| = 2^{-(2k+3)}$.
 \end{itemize}
 Finally, 
  \begin{itemize}
   \item[(vi)] $|\mu_{(-1,-1),(0,0)}| = a \, 2^{-(n+3)}  + 2^{-2(n+1)}$ with $a=4$ if $n$ is even and $a=3$ if $n$ is odd.
  \end{itemize} 
\end{lemma}

Then using these Haar coefficients in Parseval's equality 
$$ D_N({\cal R}_n, L_2)^2 = \sum_{\bsj\in \NN_{-1}^2} \sum_{\bsm\in {\mathbb D}_{\bsj}} \frac{\mu_{\bsj,\bsm}^2}{\|h_{\bsj,\bsm}\|_2^2} $$
gives the upper bound
$$ D_N({\cal R}_n, L_2)^2 \ll \frac{n}{2^{2n}} = \frac{\log N}{ N^2}. $$

Using the Littlewood-Paley inequality, which is explained in Section \ref{LPI}, as replacement for Parseval's equality also provides optimality of the symmetrized Hammersley set for the $L_p$-discrepancy for $1<p<\infty$.
Similarly, optimality can be shown in Besov spaces of dominating mixed smoothness for certain parameter values, as these can be characterized by an equivalent norm via Haar coefficients, see \cite{hin2010,tri2010,tri2010a}. 
For generalizations to higher dimensions, see \cite{Mar2013a,Mar2013b}.
Faber bases can be used to derive error bounds in cases where the Haar functions do not work, see e.g. \cite{tri2014,ull2014}.

\subsection{Orthogonal Functions - Lower Bounds}\label{OFLB}

The crucial idea for proving lower bounds of norms of the discrepancy function is that the contribution of dyadic boxes containing no point can be amplified with the help of orthogonality.
This idea, which is due to Roth \cite{R54}, resonates with the idea of constructing ``bump functions'' which is represented in Section~\ref{subsec_bumbf}. 

\subsubsection*{Example: Roth's lower bound for the $L_2$-discrepancy}

\begin{theorem}[Roth]\label{thm:roth}
 The $L_2$-discrepancy of any $N$-element point set ${\cal P}\subseteq [0,1)^s$ satisfies the lower bound
 $$D_N({\cal P}, L_2) \gg_s  \frac{(\log N)^{(s-1)/2}}{N}. $$
\end{theorem}

\begin{proof}
Roth used, together with orthogonality, also duality and the Cauchy-Schwarz inequality. We present here a version of the proof which just uses Bessel's inequality and Haar functions.
To this end, we need the inner products of the discrepancy function with the Haar functions. The following two lemmas separately deal with the volume part ${\rm vol} (B_{\bsx})$ and the
counting part $\frac{1}{N} \sum_{\bsz\in{\cal P}} {\mathbf 1}_{B_{\bsx}}(\bsz)$ of the discrepancy  function. 
Both are easy calculations which can be reduced to the one-dimensional case using the product structure of the involved functions.

\begin{lemma}[Volume part]\label{levolpart}
 Let $\bsj=(j_1,\dots,j_s)\in{\mathbb N}_{0}^s$ and $\bsm\in {\mathbb D}_{\bsj}$. Then
 $$ \langle x_1 \cdots x_s , h_{\bsj,\bsm}(\bsx) \rangle = 2^{-2j_1-\cdots-2j_s-2}.$$
\end{lemma}

\begin{lemma}[Counting part]\label{lecountpart}
 Let $\bsj=(j_1,\dots,j_s)\in{\mathbb N}_{0}^s$ and $\bsm\in {\mathbb D}_{\bsj}$. Then
 $$ \langle {\mathbf 1}_{B_{\bsx}}(\bsz) , h_{\bsj,\bsm}(\bsx) \rangle = 0$$
 whenever $\bsz$ is not contained in the dyadic box supporting $h_{\bsj,\bsm}$
\end{lemma}

Now we choose a level $\ell$ such that $ 2^{\ell-1} < 2N \le 2^{\ell}$,  so that $\ell \approx \log N$.
Then, for each shape in level $\ell$, at least half of the $2^{\ell}$ dyadic boxes of this shape do not contain any points of ${\cal P}$.
So, in the computation of the corresponding Haar coefficients of the discrepancy function, the counting part does not count.
Let $S$ be the set of all pairs $(\bsj,\bsm)$ such that $I_{\bsj,\bsm}$ does not contain any points of ${\cal P}$ and is of level $\ell$.
We then obtain from Bessel's inequality and Lemma \ref{levolpart} that
$$ 
   D_N({\cal P}, L_2)^2 \ge \sum_{(\bsj,\bsm)\in S} 2^\ell \langle  D_N({\cal P}, \, \cdot \,), h_{\bsj,\bsm} \rangle^2 = 2^{-3\ell-4} \# S \approx_s 
	 2^{-2\ell}\ell^{s-1}
$$
proving the theorem.
\end{proof}

By taking more care of the number of empty boxes the best known lower bounds for the $L_2$-discrepancy are derived  in \cite{HM2011}.

\subsection{Littlewood-Paley Inequality}\label{LPI}

The Littlewood-Paley inequality provides a tool which can be used to replace Parseval's equality and Bessel's inequality for functions in $L_p(\RR)$ with $1<p<\infty$.
It involves the {\em square function} $S(f)$ of a function $f\in L_p([0,1))$ which is given as
$$ S(f) = \left( \sum_{j,m} 2^{2j} \, \langle f , h_{j,m} \rangle^2 \, {\mathbf 1}_{I_{j,m}} \right)^{1/2}.$$

\begin{theorem}[Littlewood-Paley inequality]\label{lpi}
 Let $1<p<\infty$ and let $f\in L_p([0,1))$. Then 
 $$ \| S(f) \|_p \approx_p \| f \|_p.$$
\end{theorem}

Proofs of these inequalities and further details also yielding the right asymptotic behavior of the involved constants can be found in \cite{bur88,ste93,WG91}.
This equivalence of norms between the function and its square function can be generalized to arbitrary dimension $s\in\NN$, see \cite{pip86,ste93}.
This leads to a short direct proof of the lower bound of Schmidt \cite{schm1977} for the $L_p$-discrepancy.

\subsubsection*{Example: Schmidt's lower bound for the $L_p$-discrepancy}

\begin{theorem}[Schmidt]\label{thm:schmidt}
 Let $1<p<\infty$. The $L_p$-discrepancy of any $N$-element point set ${\cal P} \subseteq [0,1)^s$ satisfies the lower bound
 $$ 	D_N({\cal P}, L_p) \gg_{s,p}  \frac{(\log N)^{(s-1)/2}}{N}. $$
\end{theorem}

\begin{proof}
 Of course, for $p\ge 2$ this follows immediately from Roth's Theorem \ref{thm:roth}.
 For the general case we proceed as in the proof of that theorem but we use the Littlewood-Paley inequality instead of Bessel's inequality and obtain
 $$ 
   D_N({\cal P}, L_p)^p \gg_{s,p} \int_{[0,1)^d} \Big|\sum_{(\bsj,\bsm)\in S}  2^{2\ell} \langle  D_N({\cal P}, \, \cdot \,), h_{\bsj,\bsm} \rangle^2  {\mathbf 1}_{I_{\bsj,\bsm}}(\bsx) \Big|^{p/2} \,\mathrm{d} \bsx.
 $$
 Now the Haar coefficients from Lemma~\ref{levolpart} and Lemma~\ref{lecountpart} show that  
 $$\langle  D_N({\cal P}, \, \cdot \,), h_{\bsj,\bsm} \rangle = 2^{-2\ell-2}$$
 which implies
 $$
   D_N({\cal P}, L_p)^p \gg_{s,p} 2^{-(\ell+2)p} \int_{[0,1)^d} \Big( \sum_{(\bsj,\bsm)\in S} {\mathbf 1}_{I_{\bsj,\bsm}}(\bsx) \Big)^{p/2} \,\mathrm{d} \bsx.
 $$
 Now observe that for each fixed $\bsj$, the sum $\sum_{\bsm:(\bsj,\bsm)\in S} {\mathbf 1}_{I_{\bsj,\bsm}}(\bsx) $ is the indicator function of a set of measure at least $\frac{1}{2}$.
 Hence 
 $$ \sum_{(\bsj,\bsm)\in S} {\mathbf 1}_{I_{\bsj,\bsm}}(\bsx) = \sum_{k=1}^M {\mathbf 1}_{A_k}(\bsx) $$
 where each $A_k$ has measure at least $\frac{1}{2}$ and 
 $M=\binom{\ell+s-2}{s-1} \approx_s \ell^{s-1} $ is the number of different shapes of boxes with level $\ell$.
 But then $ \sum_{k=1}^M {\mathbf 1}_{A_k}(\bsx) \ge \frac{M}{4} $ on a set of measure at least $\frac14$, so that we obtain
 $$ 
  D_N({\cal P}, L_p)^p \gg_{s,p} 2^{-(\ell+2)p} \, \frac{1}{4} \, \left( \frac{M}{4} \right)^{p/2} \gg_{s,p} \left( \frac{(\log N)^{(s-1)/2}}{N} \right)^p
 $$
 proving the theorem.
\end{proof} 

The Littlewood-Paley decomposition lends itself to the analysis of functions in further function spaces in harmonic analysis like $BMO$ and $\exp(L^\alpha)$, see \cite{BLPV},
Hardy spaces $H_p$ for $0<p<1$, see \cite{lac2008},
and spaces of dominating mixed smoothness, see \cite{hin2010,Mar2013,Mar2013a,Mar2013b}. 
A recent survey of Roth's method and its extensions is \cite{bil2011}.

\subsection{Riesz Products}\label{RP}

The Littlewood-Paley approach from the previous section is not directly applicable to the endpoints $p=1,\infty$.
But {\em Riesz products}, another tool from harmonic analysis, can be used to prove sharp lower bounds in the case $p=1$ and $s=2$.
This approach is due to Hal\'asz \cite{hal1981}.

\subsubsection*{Example: Hal\'asz' lower bound for the $L_1$-discrepancy}

\begin{theorem}[Hal\'asz]\label{thm:halasz}
 The $L_1$-discrepancy of any $N$-element point set ${\cal P}\subseteq [0,1)^2$ satisfies the lower bound
 $$ 	D_N({\cal P}, L_1) \gg  \frac{\sqrt{\log N}}{N}. $$
\end{theorem}
 
\begin{proof}[Sketch]
 We again start as in the proof of Theorem \ref{thm:roth} and choose a level $\ell$ such that $ 2^{\ell-1} < 2N \le 2^{\ell}$,  so that $\ell \approx \log N$.
 Observe that the shape of a rectangle in level $\ell$ is now fixed by the parameter $j=j_1$ fixing the size in the first coordinate direction. 
 Now for each such $j=0,1,\dots,\ell$ we add up the Haar functions of all dyadic rectangles $I_{(j,\ell-j),\bsm}$ which do not contain points of ${\cal P}$
 and obtain orthogonal functions $f_0,f_1,\dots,f_\ell$ which only take values $\pm 1$ and $0$. Moreover, since we add up at least $2^{\ell-1}$ such Haar functions, 
 we obtain  from Lemma~\ref{levolpart} and Lemma~\ref{lecountpart} that  
 $$ \langle  D_N({\cal P}, \, \cdot \,), f_j \rangle \ge 2^{\ell-1}  2^{-2\ell-2} = 2^{-\ell-3} \approx \frac1N.$$
 These functions are now used to build up the Riesz product
 $$ F := \prod_{j=0}^\ell \left( 1 + \frac{\icomp c }{ \sqrt{\ell+1} } f_j \right) - 1 = \frac{\icomp c }{ \sqrt{\ell+1}} \sum_{j=0}^\ell f_j + R $$
 with some small $c>0$. Here the function $R$ collects all the products of two and more Haar functions involved. 
 It follows that
 $$ \left| \langle  D_N({\cal P}, \, \cdot \,), F \rangle \right| \ge c \frac{\sqrt{\ell+1}}{2^{l+3}} - \left| \langle  D_N({\cal P}, \, \cdot \,), R \rangle \right|. $$
 Now the property that arbitrary products of the Haar functions involved are again Haar functions on a higher level, one can show that 
 $\left| \langle  D_N({\cal P}, \, \cdot \,), R \rangle \right|$ is small compared with $c \ \frac{\sqrt{\ell+1}}{2^{l+3}}$ if $c$ is chosen sufficiently small, but independent of $N$.
 The second crucial property of $F$ is that
 $$ \|F\|_\infty \le \left| 1 + \frac{\icomp c }{ \sqrt{\ell+1} }  \right|^{\ell+1} +1 =  \left( 1 + \frac{c^2 }{ \ell+1 }  \right)^{\frac{\ell+1}{2}} +1 \le \exp\left(\tfrac{c^2}{2}\right) + 1,$$
 which motivates the use of complex numbers. 
 It follows that
 $$ D_N({\cal P}, L_1) \ge \frac{\left| \langle  D_N({\cal P}, \, \cdot \,), F \rangle \right|}{\|F\|_\infty} \gg \frac{\sqrt{\ell+1}}{2^\ell} \gg \frac{\sqrt{\log N}}{N}. $$
\end{proof}

The proof of Hal\'asz provides the sharp lower bound for the $L_1$-discrepancy in dimension $s=2$. The same bound is the best known lower bound
also for higher dimensions. It is one of the main open problems in discrepancy theory to improve this lower bound. 
Also the correct lower bound of Schmidt \cite{schm1972} for the $L_\infty$-discrepancy in dimension $s=2$ can be proved with this method as was demonstrated by Halasz \cite{hal1981}.

\section{Algebra and Number Theory}\label{sec_alg_num}

Algebra and number theory enter the stage of QMC through the various constructions of point sets with good equidistribution properties, which are required as sample nodes for QMC algorithms, and their analysis. Almost all constructions of point sets and sequences relevant for QMC are based on number theoretic or algebraic concepts. 

\subsection{Lattices}\label{secLPS}





Lattices are an important concept in number theory, especially in the geometry of numbers which play also an important role in the construction of point sets and QMC rules.
\begin{definition}
A lattice $L$ in $\RR^s$ is a discrete subset of $\RR^s$ which is closed under addition and subtraction. 
\end{definition}
Note that a lattice contains the origin. For every lattice $L$ in $\RR^s$ there exists a lattice basis which is a set $\{\bsw_1,\bsw_2,\ldots,\bsw_s\}$ of linearly independent vectors such that the lattice consists exactly of all integer linear combinations of $\bsw_1,\bsw_2,\ldots,\bsw_s$. The $s \times s$ matrix $W$ with rows $\bsw_1,\bsw_2,\ldots,\bsw_s$ is called the generator matrix of $L$ and the determinant of $L$ denoted by $\det(L)$ is the absolute value of the determinant of the generator matrix $W$. We note that the lattice bases, and therefore also $W$, are not uniquely determined but it can be shown that $\det(L)$ is an invariant for the lattice $L$.  

Information on lattice rules in the context of QMC can be found in \cite{LP,niesiam,slojoe}. In the following we present the two basic examples.

\subsubsection*{Example: General lattice rules}
For $\bsx, \bsy \in \mathbb{R}^s$ we say that the equivalence relation $\bsx \sim \bsy$ holds iff there exists some $\bsz \in \ZZ^s$ such that $\bsx=\bsy+\bsz$. We define the equivalence classes $\bsx + \mathbb{Z}^s = \{\bsx + \bsz \in \mathbb{R}^s: \bsz \in \mathbb{Z}^s\}$. By $\RR^s/\ZZ^s$ we denote the set of all equivalence classes $\bsx+\mathbb{Z}^s$ of $\RR^s$ modulo $\ZZ^s$, equipped with the addition $(\bsx+\ZZ^s)+(\bsy+\ZZ^s):= (\bsx + \bsy)+\ZZ^s$, where $\bsx+\bsy$ denotes the usual addition in $\RR^s$. With these definitions $\RR^s/\ZZ^s$ becomes an abelian group.
 
Let $L/\ZZ^s$ be any finite subgroup of $\RR^s/\ZZ^s$ and let $\bsx_n+\ZZ^s$ with $\bsx_n \in [0,1)^s$ for $n=0,1,\ldots,N-1$ be the distinct residue classes which form the group $L/\ZZ^s$. Then the set $\{\bsx_0,\bsx_1,\ldots,\bsx_{N-1}\}$ is said to be the node set of the lattice rule $L$. If we view $L=\bigcup_{n=0}^{N-1} (\bsx_n+\ZZ^s)$ as a subset of $\RR^s$, then $L$ is an $s$-dimensional lattice. 

\subsubsection*{Example: Rank-1 lattice rules}
For $N \in \NN$, $N \ge 2$, $s \in \NN$ and $\bsg \in \ZZ^s$ an $N$-element {\it rank-1 lattice point set} $\cP(\bsg,N)=\{\bsx_0,\bsx_1,\ldots,\bsx_{N-1}\}$ is defined by 
\begin{equation}\label{LPS}
\bsx_n=\left\{\frac{n}{N} \bsg\right\}\ \ \ \mbox{ for }\ n=0,1,\ldots,N-1,
\end{equation}
where the fractional part function $\{\cdot\}$ is applied component-wise. QMC rules that use rank-1 lattice point sets as underlying nodes are called {\it (rank-1) lattice rules}. The residue classes $\bsx_n+\ZZ^s=(n/N)\bsg+\ZZ^s$ for $n=0,1,\ldots,N-1$ corresponding to a lattice point set as defined in \eqref{LPS} form a finite cyclic subgroup of the additive group $\RR^s/\ZZ^s$ generated by $(1/N)\bsg+\ZZ^s$. Hence (rank-1) lattice rules are a sub-class of general lattice rules.

Rank-1 lattice point sets can also be viewed as finite versions of {\it Kronecker sequences} $\cS_{\bsalpha}=(\bsx_n)_{n \ge 0}$ which are defined as $$\bsx_n=\{n \bsalpha\}\ \ \ \mbox{ for }\ n \in \NN_0,$$ where $\bsalpha \in \RR^s$ and where the fractional part $\{\cdot\}$ is again applied component-wise. See \cite{DT,kuinie} or Section~\ref{secDioph} for more information. The discrepancy of rank-1 lattice point sets will be discussed in Section~\ref{secExpSum} and the one of Kronecker sequences in Section~\ref{secDioph}.



\subsection{Digital Constructions}\label{secDig}

Digit expansions are a basic concept in number theory which also have applications in QMC or, in more detail, in the construction of QMC points and sequences. 

Let $b \ge 2$ be an integer. Every $n \in \NN_0$ can be expanded in its $b$-adic digit expansion $n=n_0+n_1 b +n_2 b^2 +\cdots$ with $b$-adic digits $n_i \in \cZ_b$, where we set $\cZ_b:=\{0,1,\ldots,b-1\}$. A large class of constructions of QMC point sets is based on manipulations of these $b$-adic digit expansions. We remark that such constructions not only exist for $b$-adic expansions but also for more general expansions such as, e.g., Ostrowski expansions, $\beta$-adic expansions, $Q$-adic expansions, etc. However, the $b$-adic expansions are the most important ones in this context. In the following we present some examples. More information on the following examples can be found in \cite{DP10,LP,niesiam} and the references therein.

\subsubsection*{Example: Van der Corput sequences} For an integer $b \ge 2$ the {\it $b$-adic radical inverse function} $\phi_b:\NN_0 \rightarrow [0,1)$ is defined by $$\phi_b(n)=\frac{n_0}{b}+\frac{n_1}{b^2}+\frac{n_2}{b^3}+\cdots$$ whenever $n \in \NN_0$ has $b$-adic digit expansion $n=n_0+n_1 b+n_2 b^2+\cdots$ (which is of course finite) with all digits $n_j \in \cZ_b$. The {\it $b$-adic van der Corput sequence} is the one-dimensional sequence $\cS_b=(x_n)_{n \ge 0}$, where $x_n=\phi_b(n)$. This sequence is {\it the} prototype of many other digital constructions of point sets and sequences. It is well-known that the discrepancy of van der Corput sequences satisfies $D_N(\cS_b)\ll_b (\log N)/N$ (see, e.g, \cite{DP10,kuinie,LP,niesiam}). 
 
\subsubsection*{Example: Halton sequences} For $s \in \NN$, $s \ge 2$, and for integers $b_1,\ldots,b_s \ge 2$ the {\it Halton sequence $\cS_{b_1,\ldots,b_s}=(\bsx_n)_{n \ge 0}$ in bases $b_1,\ldots,b_s$} is defined by $$\bsx_n=(\phi_{b_1}(n),\ldots,\phi_{b_s}(n))\ \ \ \mbox{ for }\ n=0,1,\ldots,$$ where $\phi_b$ is the $b$-adic radical inverse function. A Halton sequence in bases $b_1,\ldots,b_s$ is uniformly distributed in $[0,1)^s$ if and only if $b_1,\ldots,b_s$ are mutually co-prime. In this case the discrepancy of the Halton sequence satisfies $D_N(\cS_{b_1,\ldots,b_s}) \ll_{s,b_1,\ldots,b_s} (\log N)^s/N$ (see, e.g., \cite{DP10,LP,niesiam}).

\subsubsection*{Example: Hammersley point sets} For $s,N \in \NN$, $s \ge 2$ and for pairwise coprime integers $b_1,\ldots,b_{s-1} \ge 2$ the $N$-element {\it Hammersley point set $\cP_{b_1,\ldots,b_{s-1}}=\{\bsx_0,\bsx_1,\ldots,\bsx_{N-1}\}$ in bases $b_1,\ldots,b_{s-1}$} is defined by $$\bsx_n:=\left(\frac{n}{N},\phi_{b_1}(n),\ldots ,\phi_{b_{s-1}}(n)\right) \ \ \mbox{ for } \ \ n=0,1,\ldots, N-1.$$ If $b_1,\ldots,b_{s-1}$ are mutually co-prime, then the discrepancy of the Hammersley point set satisfies $D_N(\cP_{b_1,\ldots,b_{s-1}}) \ll_{s,b_1,\ldots,b_s} (\log N)^{s-1}/N$ (see, e.g., \cite{DP10,LP,niesiam}).

\subsubsection*{Example: Digital nets} The construction of digital nets 
is based on finite rings $R_b$ of order $b$. Here we restrict our discussion to the case where $R_b$ is the finite field $\FF_b$ of prime-power order $b$. First one requires a bijection $\varphi:\cZ_b\rightarrow \FF_b$ and $m \times m$ matrices $C_1,\ldots,C_s$ over $\FF_b$ (one per component). A {\it digital net $\{\bsx_0,\bsx_1,\ldots,\bsx_{b^m-1}\}$ over $\FF_b$} with generating matrices $C_1,\ldots,C_s$ is constructed in the following way: for $n=0,1,\ldots,b^m-1$ write $n$ in its base $b$ expansion $n=n_0+n_1 b +\cdots +n_{m-1} b^{m-1}$ with digits $n_j \in \cZ_b$. For $j \in [s]$ compute the matrix vector product $$C_j \left(\begin{array}{c} \varphi(n_0)\\ \varphi(n_1)\\ \vdots \\ \varphi(n_{m-1})\end{array}\right)=:\left(\begin{array}{c} \overline{y}_{n,j,1}\\ \overline{y}_{n,j,2}\\ \vdots \\ \overline{y}_{n,j,m}\end{array}\right),$$ where all arithmetic operations are carried out in $\FF_b$, set $$x_{n,j}:=\frac{\varphi^{-1}(\overline{y}_{n,j,1})}{b}+\frac{\varphi^{-1}(\overline{y}_{n,j,2})}{b^2}+\cdots +\frac{\varphi^{-1}(\overline{y}_{n,j,m})}{b^m}$$ and put $$\bsx_n:=(x_{n,1},\ldots ,x_{n,s}).$$ If the order $b$ of the underlying finite field is a prime number, then one often identifies $\FF_b$ with the set $\cZ_b$ equipped with arithmetic modulo $b$. In this case it is convenient to choose the identity for the bijection $\varphi$ .

Depending on the choice of the generating matrices, digital nets can achieve a discrepancy of order $(\log N)^{s-1}/N$. We refer to \cite{DP10,LP,niesiam} for more information on the discrepancy of digital nets.

\subsubsection*{Example: Digital sequences} The construction of digital sequences over $\FF_b$ is analogous to the one of digital nets over $\FF_b$ with the difference that one requires $C_1,\ldots,C_s$ to be $\NN \times \NN$ matrices over $\FF_b$. For technical reasons the bijection $\varphi$ has to map $0$ to the zero element of $\FF_b$. For every $m \in \NN$ the initial $b^m$ elements of a digital sequence form a digital net with $b^m$ elements.

Depending on the choice of the generating matrices, digital sequences can achieve a discrepancy of order $(\log N)^{s}/N$ for all $N \ge 2$. We refer to \cite{DP10,LP,niesiam} for more information on the discrepancy of digital sequences.

\subsection{Polynomial Arithmetic and Formal Laurent Series}\label{secPafLs}

Polynomial arithmetic and formal Laurent series over a finite field play also an important role in the construction of QMC point sets and sequences. 
 
Let $b$ be a prime power and let $\FF_b$ be the finite field of order $b$. If $b$ is a prime number, then we identify $\FF_b$ with the set $\cZ_b=\{0,1,\ldots,b-1\}$ equipped with arithmetic operations modulo $b$. Let $\FF_b[x]$ be the set of all polynomials over $\FF_b$ and let $\FF_b((x^{-1}))$ be the field of formal Laurent series $$g=\sum_{k=w}^{\infty} a_k x^{-k}\ \ \ \mbox{ with } a_k \in \FF_b \ \mbox{ and } \ w \in \ZZ \ \mbox{ with }\ a_w \not=0.$$ For $g \in \FF_b((x^{-1}))$ and $m \in \NN \cup \{\infty\}$ we define the ``fractional part'' function $\FF_b((x^{-1})) \rightarrow [0,1)$ by $$\{g\}_{b,m}:= \sum_{k=\max(1,w)}^{m} a_k b^{-k}.$$ 

In the following we present some examples of constructions based on the concepts of polynomial arithmetic and formal Laurent series. More information can be found in \cite{DP10,niesiam}.

\subsubsection*{Example: Polynomial lattice point sets} 
Let $m \in \NN$ and let $b$ be a prime number. Given a $p \in \FF_b[x]$ with $\deg(p)=m$ and $\bsq=(q_1,\ldots ,q_s) \in \FF_b[x]^s$ a polynomial lattice point set $\cP(\bsq,p)$ is given by the points $$\bsx_h=\left(\left\{\frac{h q_1}{p}\right\}_{b,m},\ldots ,\left\{\frac{h q_s}{p}\right\}_{b,m}\right),$$ for $h \in \FF_b[x]$ with $\deg(h)<m$. QMC rules that use polynomial lattice point sets as underlying nodes are called {\it polynomial lattice rules}. 

Polynomial lattice point sets have been first introduced by Niederreiter \cite{nie92} and can be viewed as polynomial analogs of lattice point sets (see Section~\ref{secLPS}). They are also special instances of digital nets over $\FF_b$ where the generating matrices $C_1,C_2,\ldots,C_s$ are constructed as follows: choose $p \in \FF_b[x]$ with $\deg(p)=m \ge 1$ and let $\bsq=(q_1,\ldots ,q_s) \in \FF_b[x]^s$. For $j=1,2,\ldots,s$, consider the formal Laurent series expansions $$\frac{q_j(x)}{p(x)}=\sum_{l=w_j}^{\infty} \frac{u_l^{(j)}}{x^l} \in \FF_b((x^{-1}))$$ where $w_j \le 1$, and put $C_j=(c_{i,r}^{(j)})_{i,r=1}^m$ where the elements $c_{i,r}^{(j)}$ of the matrix $C_j$ are given by $c_{i,r}^{(j)}=u_{r+i-1}^{(j)}\in \FF_b$ for $j=1, \ldots,s$ and $i,r=1, \ldots, m$. The latter viewpoint also allows for constructions of ``polynomial lattice point sets'' in the prime-power base case.

For prime $b$ it is known that for any $p\in\FF_b [x]$ with the property $p(x)=x^m$ or $\gcd(p,x)=1$ and $\deg(p)=m$ there exists a generating vector $\bsq \in \FF_b[x]^s$ such that
\[D_N(\cP(\bsq,p))\ll_{s,b} \frac{(\log N)^{s-1} \log \log N}{N}.\] See \cite{KP12,lar93} for more information.

\subsubsection*{Example: Digital Kronecker sequences}

Let $b$ be a prime number. For every $s$-tuple $\bsf=(f_1,\ldots,f_s)$ of elements of $\FF_b((x^{-1}))$ we define the sequence $\cS(\bsf)=(\bsx_n)_{n \ge 0}$  by $$\bsx_n=(\{n f_1\}_b,\ldots,\{n f_s\}_b)\ \ \mbox{ for }\ \ \ n \in \NN_0,$$ where we associate a nonnegative integer $n$ with $b$-adic expansion $n=n_0+n_1b+\cdots+n_r b^r$ with the polynomial $n(x)=n_0+n_1 x+\cdots +n_r x^r$ in $\FF_b[x]$ and vice versa and where $\{g\}_b:=\{g\}_{b,\infty}$. The sequence $\cS(\bsf)$ can be viewed as an analogue of the classical Kronecker sequence and is therefore called a {\it digital Kronecker sequence}.

Digital Kronecker sequences are special examples of digital sequences (see Section~\ref{secDig}). Consider $\bsf=(f_1,\ldots,f_s)$ with $f_j=\frac{f_{j,1}}{x}+\frac{f_{j,2}}{x^2}+\frac{f_{j,3}}{x^3}+\cdots \in \FF_b((x^{-1}))$. Then the digital Kronecker sequence $\cS(\bsf)$ is a digital sequence generated by the $\NN \times \NN$ matrices $C_1,\ldots,C_s$ over $\FF_b$ given by $$C_j=\left(
\begin{array}{llll}
f_{j,1} & f_{j,2} & f_{j,3} & \ldots\\
f_{j,2} & f_{j,3} & f_{j,4} & \ldots\\
f_{j,3} & f_{j,4} & f_{j,5} & \ldots\\ 
\multicolumn{4}{c}\dotfill 
\end{array}\right).$$

\subsubsection*{Example: Generalized Niederreiter sequences.} These are special instances of digital sequences over $\FF_b$ where the generating matrices $C_1,C_2,\ldots,C_s$ are constructed as follows: let $p_1,\ldots, p_s \in \FF_b[x]$ be distinct monic irreducible polynomials over $\FF_b$. For each $i \in \NN$ and $j=1,2,\ldots,s$ choose a set of polynomials $\{y_{j,i,k}(x) :0 \le k < e_j\}$ which has to be linearly independent modulo $p_j(x)$ over $\FF_b$. Consider the expansion
\begin{equation*}
\frac{y_{j,i,k}(x)}{p_j(x)^i} = \sum_{r=1}^\infty \frac{a^{(j)}(i,k,r)}{x^r}
\end{equation*}
over $\FF_b((x^{-1}))$ and define the matrix $C_j= (c^{(j)}_{i,r})_{i,r \in \NN}$ by
\begin{equation*}
c_{i,r}^{(j)} = a^{(j)}(Q+1,k,r) \in \FF_b \quad\mbox{for } \ j \in [s],\ i,r \in \NN,
\end{equation*}
where $i-1 = Q e_j + k$ with integers $Q = Q(j,i)$ and $k = k(j,i)$ satisfying $0 \le k < e_j$. Generalized Niederreiter sequences comprise Sobol'-, Faure- and Niederreiter-sequences as special cases.

\subsection{Groups, Characters and Duality}\label{grchardual}

Let $(G,\circ)$ be a finite abelian group. A {\it character of $G$} is a group homomorphism $\chi:G \rightarrow \CC^{\times}$, that is, for all $x,y\in G$ we have $\chi(x \circ y)=\chi(x) \chi(y)$. This already implies $\chi(1_G)=1$, where $1_G$ is the identity in $G$. Every finite abelian group of order $N$ has exactly $N$ distinct characters denoted by $\chi_0,\chi_1,\ldots,\chi_{N-1}$ where the character $\chi_0\equiv 1$, which is 1 for all $x \in G$, is called the {\it trivial character} or the {\it principal character}. The set $\widehat{G}$ of all characters of $G$ forms an abelian group under the multiplication $(\chi\psi)(x)=\chi(x)\psi(x)$ for all $x \in G$, for $\chi,\psi \in \widehat{G}$.

Characters have the following important property which can be exploited in many applications. 

\begin{lemma}[Character properties]\label{lecharprop}
Let $\chi$ be a character of a finite abelian group $(G,\circ)$. Then we have $$\sum_{x \in G} \chi(x)=\left\{
\begin{array}{ll}
|G| & \mbox{ if $\chi$ is the trivial character},\\
0 & \mbox{ otherwise}. 
\end{array}\right.$$
Let $x \in G$. Then we have $$\sum_{\chi \in \widehat{G}} \chi(x)=\left\{
\begin{array}{ll}
|\widehat{G}| & \mbox{ if $x=1_G$},\\
0 & \mbox{ otherwise}. 
\end{array}\right.$$
\end{lemma}

\begin{proof}
We just prove the first identity, the second one follows by a similar reasoning. The result is clear when $\chi$ is the trivial character. Otherwise there exists some $a \in G$ for which we have $\chi(a) \not=1$. Then we have
$$\chi(a) \sum_{x \in G} \chi(x)=\sum_{x \in G} \chi(a \circ x)=\sum_{x \in G}\chi(x),$$ since as $x$ runs through all elements of $G$ so does $a\circ x$. Hence we have $$(\chi(a)-1)\sum_{x \in G}\chi(x)=0$$ and the result follows since $\chi(a)\not=1$. 
\end{proof}

More information on characters of finite abelian groups can be found in \cite[Chapter~5, Section~1]{LidNie}. Many constructions of QMC point sets have an inherent group structure and for these instances the above character property is an important tool for their analysis.  We present the two most important examples.

\subsubsection*{Example: General lattice rules}  

Let $L/\ZZ^s$ be any finite subgroup of $\RR^s/\ZZ^s$ and let $\{\bsx_0,\bsx_1,\ldots,\bsx_{N-1}\}$ be the node set of the lattice rule $L$ (see Section~\ref{secLPS}). Recall the definition of the $\bsk$th trigonometric functions ${\rm e}_{\bsk}: [0,1)^s \to \CC$ from Section~\ref{OBEB} given by  
\begin{equation}\label{deftrigfct}
{\rm e}_{\bsk}(\bsx)=\exp(2 \pi \icomp \bsk \cdot \bsx) \ \ \mbox{ for } \bsx \in [0,1)^s.
\end{equation}
Then $\chi_{\bsk}(\bsx+\ZZ^s)={\rm e}_{\bsk}(\bsx)$ for $\bsx \in L$ is a well-defined character of the additive group $L/\ZZ^s$. This character is trivial if and only if $\bsk \in L^\bot$, where $$L^\bot= \{\bsh \in \ZZ^s \ : \ \bsh \cdot \bsx \in \ZZ \mbox{ for all } \bsx \in L\}.$$ For rank-1 lattice point sets as defined in \eqref{LPS} it is clear that 
\begin{equation}\label{def_duallatt_rank1}
L^\bot= \{\bsh \in \ZZ^s \ : \ \bsh \cdot \bsg \equiv 0 \pmod{N}\}.
\end{equation}
The set $L^\bot$ is again a lattice in $\RR^s$ which is called the {\it dual lattice of $L$}.

Now Lemma~\ref{lecharprop} yields the following important result:
\begin{lemma}\label{explatchar}
Let $\{\bsx_0,\bsx_1,\ldots,\bsx_{N-1}\}$ be the node set of an $N$-element lattice rule $L$. Then for $\bsk \in \ZZ^s$ we have $$\sum_{n=0}^{N-1} {\rm e}_{\bsk}(\bsx_n)= \left\{
\begin{array}{ll}
N & \mbox{ if } \bsk \in L^\bot,\\
0 & \mbox{ if } \bsk \not\in L^\bot. 
\end{array}\right.$$
\end{lemma}

This basic property is exploited in the analysis of the worst-case error of lattice rules (see Section~\ref{OBEB} and \cite[Chapter~5]{niesiam}) or of discrepancy estimates of the corresponding node sets (see Section~\ref{secExpSum}).

\subsubsection*{Example: Digital nets} Let $b$ be a prime-power and let $\varphi:\cZ_b\rightarrow \FF_b$ be a bijection with $\varphi(0)=\overline{0}$ be fixed. For $x,y \in [0,1)$ let $x=\frac{\xi_{1}}{b}+\frac{\xi_{2}}{b^{2}}+\cdots$ and $y=\frac{\eta_{1}}{b}+\frac{\eta_{2}}{b^{2}}+\cdots$ be their $b$-adic expansions (with $\xi_{i} \not= b-1$ for infinitely many $i$ and $\eta_{j} \not = b-1$ for infinitely many $j$). Then $x \oplus y := \frac{\zeta_{1}}{b}+\frac{\zeta_{2}}{b^{2}}+\cdots$ with \[\zeta_{j} = \varphi^{-1}(\varphi(\xi_{j}) +\varphi(\eta_{j})) \;\;\mbox{ for }\;\; j \in \NN.\] (A case which has to be excluded is, for instance (for prime $b$, $\cZ_b=\FF_b$ and $\varphi={\rm id}$), when $x = (b-1)(b^{-1} + b^{-3} + b^{-5} + \cdots)$ and $y = (b-1) (b^{-2} + b^{-4} + b^{-6} + \cdots)$. In this case $x \oplus y = (b-1) (b^{-1} + b^{-2} + b^{-3} + \cdots) = 1$.) For vectors $\bsx,\bsy \in [0,1)^s$ the $b$-adic addition $\bsx \oplus \bsy$ is defined component-wise. Note that in this way $\oplus$ is defined for almost all $\bsx,\bsy \in [0,1)^s$. 

Let $\Dcal=\{\bsx_0,\bsx_1,\ldots,\bsx_{b^m-1}\}$ be a digital net over $\FF_b$ with $m \times m$ generating matrices $C_1,\ldots ,C_s$ as defined in Section~\ref{secDig}. Any vector $\bfn=(\overline{n}_{0},\overline{n}_1,\ldots, \overline{n}_{m-1})^{\top} \in \FF_b^m$ uniquely represents an integer $n := n_{0}+ n_{1}b + \dots + n_{m-1}b^{m-1}$ from $\{0,
\ldots ,b^{m}-1\}$ via $n_i =\varphi^{-1}(\overline{n}_i)$ for $i=0,1,\ldots, m-1$, and to any such integer belongs an element $\bsx_n$ of $\Dcal$. Then the mapping
\begin{eqnarray*}
\Psi: \FF_b^{m} \rightarrow \Dcal,\ \ \bfn \mapsto  \bsx_n
\end{eqnarray*}
is a group-isomorphism from the additive group of $\FF_b^{m}$ to $\Dcal$. In fact, for $\bfn,\bfl \in \FF_b^m$ the property $\Psi(\bfn +
\bfl)=\Psi(\bfn) \oplus \Psi(\bfl)$ easily follows from the fact that for any $m \times m$ matrix $C$ over $\FF_b$ we have $C(\bfn + \bfl) = C\bfn + C \bfl$. Therefore we have:

\begin{lemma}\label{lem3.42}
Any digital net $(\Dcal,\oplus)$ is a finite abelian group.
\end{lemma}

For the sake of simplicity let in the following $b$ be a prime number and identify the finite field $\FF_b$ with $\cZ_b$ and choose $\varphi={\rm id}$.

For $k \in \NN_0$ with $b$-adic expansion $k=\kappa_0+\kappa_1 b+\kappa_2 b^2+\cdots$, where $\kappa_i \in \cZ_b$, the {\it $k$th $b$-adic Walsh function} $\walb_k: [0,1) \rightarrow \CC$ is defined as $$\walb_k(x)=\exp(2 \pi \icomp (\kappa_0  \xi_1+\kappa_1 \xi_2+\kappa_2 \xi_3+\cdots)/b),$$ for $x \in[0,1)$ with $b$-adic expansion $x=\xi_1 b^{-1}+\xi_2 b^{-2}+\xi_3 b^{-3}+\cdots$ (unique in the sense that infinitely many of the digits $\xi_i$ must be different from $b-1$). For vectors $\bsk = (k_1, \ldots, k_s) \in \NN_0^s$ and $\bsx = (x_1,\ldots,x_s) \in [0,1)^s$ we write
\[
   \walb_{\bsk}(\bsx) :=\walb_{k_1,\ldots,k_s}(x_1,\ldots,x_s)=\prod_{j=1}^s \walb_{k_j}(x_j).
\]
The system $\{\walb_{\bsk} \, : \, \bsk \in \NN_0^s\}$ is called the {\it $s$-dimensional $b$-adic Walsh function system}.\\
For all $\bsx,\bsy \in [0,1)^s$, for which $\bsx \oplus \bsy$ is defined we have $$\walb_{\bsk}(\bsx) \walb_{\bsk}(\bsy) =\walb_{\bsk}(\bsx \oplus \bsy)\ \ \mbox{ for all } \ \bsk \in \NN_0^s.$$ In particular, $\walb_{\bsk}$ is a character of the finite abelian group $(\Dcal,\oplus)$. For $\bsk=(k_1,\ldots,k_s) \in \NN_0^s$ we have $\walb_{\bsk}(\bsx_n)=1$ for all $n=0,1,\ldots, b^m-1$ if and only if $$\sum_{j=1}^s \bfk_j \cdot \bfx_{n,j}=0 \mbox{ for all } n=0,1,\ldots ,b^m-1,$$ where $\bfk_j$ is the $m$-dimensional column vector of $b$-adic digits of $k_j$ and $\bfx_{n,j}$ denotes the $m$-dimensional column vector of $b$-adic digits of the $j$th component of $\bsx_n$. From the construction of the digital net we find that $\bfx_{n,j}=C_j \bfn$, where $\bfn$ denotes the column vector of $b$-adic digits of $n$, and hence $\walb_{\bsk}(\bsx_n)=1$ for all $n=0,1,\ldots, b^m-1$ if and only if $$\sum_{j=1}^s\bfk_j\cdot  C_j \bfn=0\ \ \mbox{ for all } \ n=0,1,\ldots, b^m-1.$$ This is satisfied if and only if $$ 
C_1^{\top} \bfk_1+\cdots +C_s^{\top} \bfk_s=\bszero.$$ Thus we have shown that $\walb_{\bsk}$ is a trivial character of $\Dcal$ if and only if $\bsk \in \Dcal^\bot$, where $$\Dcal^\bot= \{\bsk \in \{0,\ldots,b^m -1\}^s \,:\, C_1^\top \bfk_1 + \cdots
+ C_s^\top \bfk_s = \bszero\}.$$ The set $\Dcal^\bot$ is called the {\it dual net of the digital net $\Dcal$}. 

Now Lemma~\ref{lecharprop} yields the following important result:
\begin{lemma}\label{charprop}
Let $b$ be a prime number and let $\Dcal$ be a digital net over $\FF_b$. Then for $\bsk \in \{0,\ldots ,b^{m}-1\}^s$ we have 
 $$\sum_{n=0}^{b^m -1} \walb_{\bsk}(\bsx_n)=\left\{
\begin{array}{ll}
b^m & \mbox{ if } \bsk \in \Dcal^\bot,\\
0 & \mbox{ if } \bsk \not\in \Dcal^\bot.
\end{array}\right.$$
\end{lemma}
This basic property is exploited in the analysis of the worst-case error of QMC rules based on digital nets or of discrepancy estimates (see, e.g., \cite{DP05,DP05b,DP10}). This can in turn be compared to the Fourier representation of the error of rank-1 lattice rules as in Section~\ref{OBEB}. For example Lemma~\ref{charprop} leads to a very concise formula for the discrepancy function of digital nets $\Dcal$ of the form $$D_N(\Dcal,\bsy)=\sum_{\bsk \in \Dcal^\bot\setminus \{\boldsymbol{0}\}} \widehat{1}_{[\boldsymbol{0},\boldsymbol{y})}(\bsk),$$ where the sum is over the dual net without zero and where  $\widehat{1}_{[\boldsymbol{0},\boldsymbol{y})}(\bsk)=\int_{[0,1]^s} 1_{[\boldsymbol{0},\boldsymbol{y})}(\bsx) \overline{\walb_{\bsk}(\bsx)}\, {\rm d} \bsx$ is the $\bsk$th Walsh-Fourier coefficient of the box-indicator function $1_{[\boldsymbol{0},\boldsymbol{y})}$. We refer to \cite[Lemma~14.8]{DP10} for a formula for $\widehat{1}_{[\boldsymbol{0},\boldsymbol{y})}(\bsk)$ in terms of Walsh series and \cite[Lemma~3.29]{DP10} for an estimate of $|\widehat{1}_{[\boldsymbol{0},\boldsymbol{y})}(\bsk)|$.

A generalization of Lemma~\ref{charprop} to the case of digital nets over $\FF_b$ with prime-power $b$ can be found in \cite[Lemma~2.5]{PDP}. In this case one requires the more general concept of Walsh functions over the finite field $\FF_b$.

\subsection{Minkowski's Fundamental Theorem}

Methods from the geometry of numbers play an important role in the analysis of lattice point sets. One of the most fundamental theorems in this area is due to Minkowski from 1896. 

\begin{theorem}[Minkowski]\label{thmink}
Let $L$ be a lattice in $\RR^s$. Then any convex set in
$\RR^s$ which is symmetric with respect to the origin and with
volume greater
than $2^s \det(L)$ contains a non-zero lattice point of $L$. 
\end{theorem}

See Cassels \cite{cass} for a proof and for more information regarding this theorem. In the following we give an application of Minkowski's result to the enhanced trigonometric degree of lattice rules. In Section~\ref{secDioph} we will apply Minkowski's theorem in the context of Diophantine approximation.

\subsubsection*{Example: The enhanced trigonometric degree of lattice rules} A cubature rule is said to have {\it trigonometric degree} $d$, if it integrates correctly all $s$-dimensional trigonometric polynomials of degree $d$. The {\it enhanced trigonometric degree} is the trigonometric degree increased by one. It is known (see \cite{Lyn2003}) that the enhanced trigonometric degree of a lattice rule generated by $\bsg \in \ZZ^s$ and consisting of $N$ nodes is $$\rho(\bsg,N)=\min_{\bsh \in L^{\bot}\setminus \{\bszero\}}|\bsh|,$$ where $|\bsh|$ is the one-norm of the vector $\bsh \in \ZZ^s$ and where $L^\bot$ is the corresponding dual lattice as defined in Section~\ref{grchardual}.

\begin{theorem}
For all $\bsg \in \ZZ^s$ and integers $N\ge 2$ we have $\rho(\bsg,N)\le (s!N)^{1/s}$. 
\end{theorem}

\begin{proof}
Let $L$ be an integration lattice generated by $\bsg \in \ZZ^s$ yielding
an $N$-point lattice rule and let $L^{\bot}$ be the dual lattice.
According to \cite[Theorem~5.30]{niesiam} we have $\det(L^{\bot})=N$.

Now consider the convex region
$$C_{\rho}^s=\{\bsx \in \RR^s \, : \,
|x_1|+\cdots+|x_s| \le \rho\},$$ where
$\rho>0$. Then $C_{\rho}^s$ is symmetric
with respect to the origin and the volume of
$C_{\rho}^s$ is
$${\rm Vol}(C_{\rho}^s)=\frac{2^s \rho^s}{s!}.$$ Hence, by
Minkowski's theorem applied to $L^\bot$, we have that if
$$\frac{2^s \rho^s}{s!} \ge 2^s \det(L^\bot) =2^s N,$$ i.e., if
$\rho \ge (s! N)^{1/s}$, then $C_{\rho}^s$ contains a non-zero
point from $L^\bot$. In other words, $L^\bot$ contains a non-zero lattice point which
belongs to $C_{(s! N)^{1/s}}^s$ and therefore we have
$\rho(\bsg,N) \le (s! N)^{1/s}$. 
\end{proof}

\subsection{Exponential Sums}\label{secExpSum}

{\it Exponential sums} are objects of the form $$S(X,F)=\sum_{x \in X} \exp(2 \pi \icomp F(x))$$ where $X$ is an arbitrary  finite set and $F$ is a real valued function on $X$. They lie at the interface of number theory and harmonic analysis and have important applications in many branches of mathematics. For example, the famous Weyl criterion (see, e.g., \cite{DP10, DT,kuinie}) states that a sequence $\cS=(\bsx_n)_{n \ge 0}$ of points in $[0,1)^s$ is uniformly distributed modulo one if and only if for all $\bsh \in \ZZ^s \setminus \{\bszero\}$ and $F_{\bsh}(\bsx)= \bsh\cdot \bsx$ we have  $$S(\cP_N,F_{\bsh}) =o(N) \ \ \mbox{ for }\ N \rightarrow \infty,$$ where $\cP_N$ is the point set consisting of the first $N$ terms of $\cS$. A quantitative version of this result is the inequality of Erd\H{o}s-Tur\'{a}n-Koksma.

\begin{theorem}[Erd\H{o}s-Tur\'{a}n-Koksma]\label{thETH}
For the discrepancy of every $N$-element point set $\cP_N$ in $[0,1)^s$ we have $$D_N(\cP_N) \ll_s \frac{1}{m}+\sum_{0< |\bsh|_{\infty} \le m}\frac{1}{r(\bsh)} \frac{|S(\cP_N,F_{\bsh})|}{N},$$ where $m \in \NN$ and where $r(\bsh)=\prod_{j=1}^s\max(1,|h_j|)$ and $|\bsh|_{\infty}=\max_{j=1,\ldots,s} |h_j|$ for $\bsh=(h_1,\ldots,h_s)\in \ZZ^s$.
\end{theorem}
A proof of this theorem can be found in \cite{DT} (and also in \cite{kuinie}, but there only for the one-dimensional case).

We present two examples which are based on the Erd\H{o}s-Tur\'{a}n-Koksma inequality and which illustrate the power of exponential sums for estimating discrepancy. More information on exponential sums can be found in \cite{LidNie,Shp,Wint}. 

\subsubsection*{Example: The star discrepancy of lattice point sets}

Combining Lemma~\ref{explatchar} and Theorem~\ref{thETH} with $m=N$ we find that the discrepancy of a rank-1 lattice point set $\cP(\bsg,N)$ (cf. Section~\ref{secLPS}) satisfies 
\begin{equation}\label{ETKlps}
D_N(\cP(\bsg,N)) \ll_s \frac{1}{N}+ R(\bsg,N),
\end{equation}
where $$R(\bsg,N):=\sum_{0< |\bsh|_{\infty} \le N \atop \bsh \in L^{\bot}}\frac{1}{r(\bsh)}.$$ For simplicity let $N$ be a prime number. We average $R(\bsg,N)$ over all $\bsg \in G_N^s$, where $G_N:=\{1,\ldots,N-1\}$, and obtain
\begin{align*}
\frac{1}{(N-1)^s} \sum_{\bsg \in G_N^s} R(\bsg,N) = & \frac{1}{(N-1)^s} \sum_{0 < |\bsh|_{\infty} \le m} \frac{1}{r(\bsh)} \sum_{\bsg \in L^{\bot} \cap G_N^s} 1.  
\end{align*}
Now $\bsg \in L^{\bot} \cap G_N^s$ means in particular that $g_1 h_1+\cdots +g_s h_s \equiv 0 \pmod{N}$. If at least one of the $h_i$'s is different from zero, then there are at most $(N-1)^{s-1}$ elements $(g_1,\ldots,g_s) \in G_N^s$ which satisfy this condition. Hence we find that
\begin{align}\label{bdavR}
\frac{1}{(N-1)^s} \sum_{\bsg \in G_N^s} R(\bsg,N) \le & \frac{1}{N-1} \sum_{0 < |\bsh|_{\infty} \le m} \frac{1}{r(\bsh)}\nonumber \\
= & \frac{1}{N-1} \left(-1+\left(\sum_{h=-N}^N \frac{1}{\max(1,|h|)}\right)^s\right)\ll  \frac{(\log N)^s}{N}.
\end{align}
Combining \eqref{ETKlps} and \eqref{bdavR} we obtain the following result.
\begin{theorem}
For every prime number $N$ there exists a lattice point $\bsg \in G_N^s$ such that $$D_N(\cP(\bsg,N)) \ll_s \frac{(\log N)^s}{N}.$$ 
\end{theorem}
For a more general and accurate result we refer to the book by Niederreiter~\cite[Chapter~5]{niesiam}. The currently best result for the discrepancy of rank-1 lattice point sets was proved by Larcher~\cite{lar86} for dimension $s=2$ and by Bykovskii~\cite{byko} for arbitrary dimension $s$.
\begin{theorem}[Bykovskii, Larcher]
For every integer $N \ge 3$, there exists a lattice point $\bsg \in \ZZ^s$ such that  
\begin{equation*}
D_N(\cP(\bsg,N)) \ll_s \frac{(\log N)^{s-1}\log \log N}{N}.
\end{equation*}
\end{theorem}

\subsubsection*{Example: Gauss sums and linear congruential pseudorandom numbers}

Discrepancy is a measure for the deviation of the distribution of a given point set from perfect uniform distribution.  Hence it is also an important test criterion for pseudorandom numbers which are required for Monte Carlo integration. The following example is taken from \cite{Wint}.

Let $b$ be a prime number. Let $\chi$ be a multiplicative character and $\psi$ be an additive character of $\FF_b$. Then $$G(\chi,\psi)=\sum_{c \in \FF_b^\times} \chi(c) \psi(c)$$ is called a {\it Gauss sum of type I}. Here and in the following  $\FF_b^\times$ denotes the multiplicative group of $\FF_b$.

Let denote by $\psi_0$ and $\chi_0$ the trivial additive and multiplicative 
character of $\FF_b$, respectively, that is $\psi_0(x)=1$ for all $x\in \FF_b$, and 
$\chi_0(x)=1$ for all $x\in \FF_b^\times$.

\begin{lemma}\label{leGI}
We have  $$G(\chi,\psi)=\left\{
\begin{array}{rl}
-1 & \mbox{ if } \chi=\chi_0 \mbox{ and } \psi\not=\psi_0,\\
0 & \mbox{ if } \chi\not=\chi_0 \mbox{ and } \psi=\psi_0,\\
q-1 & \mbox{ if } \chi=\chi_0 \mbox{ and } \psi=\psi_0. 
\end{array}\right.$$
If $\chi$ and $\psi$ are both nontrivial, then $|G(\chi,\psi)| =\sqrt{b}$.
\end{lemma}

\begin{proof}
We only show the case where $\chi$ and $\psi$ are both nontrivial. Then we have
\begin{align*}
|G(\chi,\psi)|^2 = G(\chi,\psi) \overline{G(\chi,\psi)} = \sum_{c,d\in \FF_b^\times} \chi(c d^{-1}) \psi(c-d) = \sum_{e \in \FF_b^\times} \chi(e) \sum_{c \in \FF_b^\times} \psi(c(1-e^{-1})).
\end{align*}
From Lemma~\ref{lecharprop} we obtain $$\sum_{c \in \FF_b^\times} \psi(c(1-e^{-1}))=\left\{
\begin{array}{rl}
b-1 & \mbox{ if } e=1,\\
-1 & \mbox{ if } e \not=1, 
\end{array}\right.$$
and hence $$|G(\chi,\psi)|^2=  b-1-\sum_{e \in \FF_b^\times\setminus\{1\}} \chi(e)=b,$$ again according to Lemma~\ref{lecharprop}. Hence $|G(\chi,\psi)| =\sqrt{b}$.
\end{proof}
(We remark that Lemma~\ref{leGI} holds even if $b$ is a prime-power.) More information on Gauss sums of type I can be found in \cite[Chapter~5, Section~2]{LidNie}.

For $n \in \NN$ a sum of the form $$S_n(\psi) =\sum_{c \in \FF_b^{\times}} \psi(c^n)$$ is called a {\it Gauss sum of type II}. Since $S_n=S_{\gcd(n,b-1)}$ we may restrict ourselves to divisors $n$ of $b-1$.

\begin{lemma}\label{leGII}
Let $n|(b-1)$. If $\psi\not=\psi_0$ then we have  $|S_n(\psi)| \le (n-1) \sqrt{b} +1.$
\end{lemma}

\begin{proof}
Let $\chi$ be a multiplicative character of $\FF_b$ of order $n$, i.e., $n$ is the least positive integer such that $\chi^n(x)=1$ for all $x \in \FF_b^\times$. Then for $x \in \FF_b^\times$ we have $$\sum_{j=0}^{n-1} \chi^j(x)=\left\{
\begin{array}{ll}
n & \mbox{ if } \chi(x)=1,\\
0 & \mbox{ otherwise}. 
\end{array}\right.$$
Since the order of $\chi$ is $n$ it follows that $\chi(x)=1$ if and only if $x=c^n$ for some $c \in \FF_b^\times$. Note that for given $x \in \FF_b^\times$ the equation $x=c^n$ has zero or exactly $\gcd(n,b-1)=n$ solutions $c \in \FF_b^\times$, since $n|(b-1)$. Then we have
\begin{align*}
|S_n(\psi)| & = \left|\sum_{c \in \FF_b^\times} \psi(c^n) \right|= \left|n \sum_{x \in \FF_b^\times \atop \chi(x)=1} \psi(x)\right| \\ & = \left| \sum_{x \in \FF_b^\times} \sum_{j=0}^{n-1} \chi^j(x) \psi(x)\right|= \left|\sum_{j=0}^{n-1}G(\chi^j,\psi)\right|\le (n-1) \sqrt{b}+1,
\end{align*}
where we used Lemma~\ref{leGI}.
\end{proof}

Now we apply Gauss sums of type II to linear congruential pseudorandom numbers.

\begin{definition}\rm
A sequence given by the recursion $$x_{n+1}=a x_n +c\ \ \ \mbox{ for }\ n \in \NN_0,$$  where $x_0,a,c \in \FF_b$ with $a \not\in \{0,1\}$ and all algebraic operations carried out in $\FF_b$ is called a {\it linear congruential pseudorandom number generator}. 
\end{definition}
Since $a\not=1$, the elements $x_n$ are given explicitly by the formula 
\begin{equation}\label{LCG_expl}
x_n=a^n x_0+\frac{a^n-1}{a-1} c \ \ \ \mbox{ for }\ n \in \NN_0.
\end{equation}
If $c \not=(1-a) x_0$, then the sequence $(x_n)_{n \ge 0}$ is $T$-periodic, where $T$ is the order of $a \pmod{b}$.

Consider now the $T$-element point set $\cP_T=\{x_0/b,x_1/b,\ldots,x_{T-1}/b\}$ in $[0,1)$ derived from a linear congruental pseudorandom number generator. Here $\mathbb{F}_b$ is identified with the integers $\{0,1,\ldots,b-1\}$. For simplicity we assume that $c=0$. Then it follows from \eqref{LCG_expl} that $$|S(\cP_T,F_h)|=\left|\sum_{n=0}^{T-1} \exp(2 \pi \icomp x_0 h a^n/b)\right|.$$ Let $w$ be a primitive root modulo $b$ and let $a=w^i$. Then we have $$T={\rm ord}_{\FF_b^\times}(a)={\rm ord}_{\FF_b^\times}(w^i)=\frac{b-1}{\gcd(b-1,i)}$$ and hence $\frac{b-1}{T}|i$. For fixed $n \in \{0,1,\ldots,T-1\}$ we have $a^n=w^{in}=w^{k \frac{b-1}{T}}$ if and only if $k \frac{b-1}{T} \equiv i n \pmod{b-1}$. Since $\frac{b-1}{T}|i$, the last congruence has exactly $\frac{b-1}{T}$ incongruent solutions $k$ modulo $b-1$. This shows that for fixed $n$ there are exactly $\frac{b-1}{T}$ different $x \in \FF_b^\times$ such that $a^n=x^{\frac{b-1}{T}}$. Therefore $$|S(\cP_T,F_h)|=\frac{T}{b-1}\left|\sum_{x \in \FF_b^\times} \exp(2 \pi \icomp x_0 h x^{\frac{b-1}{T}}/b)\right|.$$ The last exponential sum is a Gauss sum of type II and hence we can apply Lemma~\ref{leGII} 
and obtain 
\begin{equation}\label{bd_GII_cong}
|S(\cP_T,F_h)| \le \sqrt{b}
\end{equation}
whenever $h \not\equiv 0 \pmod{b}$. We remark that the bound on the Gauss sum of type II is only nontrivial if $T>\sqrt{b}$. 
However, there are several nontrivial estimates known for smaller $T$. In particular in \cite{BGK} the authors proved nontrivial bounds for any $T\ge b^\delta$ and $\delta>0$. 

Inserting the estimate \eqref{bd_GII_cong} into the Erd\H{o}s-Tur\'{a}n-Koksma inequality (Theorem~\ref{thETH}) finally we obtain $$D_T(\cP_T) \ll \sqrt{b} \frac{\log T}{T}.$$

Dealing with incomplete Gauss sums Niederreiter~\cite[Theorem~1]{nie74} showed a more general result which considers also parts of the  period.

\begin{theorem}[Niederreiter]
For the sequence $\cS=\{x_n/b\ : \ n=0,1,\ldots,N-1\}$ where $x_n$ are linear congruental pseudorandom numbers, $N <T$ and $T$ is the order of $a$, we have $D_N(\cS) \ll \sqrt{b}(\log b)^2/N$. 
\end{theorem}

\subsection{$b$-adic Numbers}

Within this section let $b \ge 2$ be a prime number. The set of $b$-adic numbers is defined as the set of formal sums
\begin{equation*}
\mathbb{Z}_b = \left\{z = \sum_{r=0}^\infty z_r b^r\, : \, z_r \in \{0,\ldots,b-1\} \mbox{ for all } r \in \NN_0\right\}.
\end{equation*}
The set $\mathbb{N}_0$ of nonnegative integers is a subset of $\mathbb{Z}_b$. For two nonnegative integers $y, z \in \mathbb{Z}_b$, the sum $y+z \in \mathbb{Z}_b$ is defined as the usual sum of integers. The addition can be extended to all $b$-adic numbers with the addition carried out in the usual manner. For instance, the inverse of $1 \in \mathbb{Z}_b$ is given by the formal sum
\begin{equation*}
(b-1) + (b-1) b + (b-1) b^2 + \cdots.
\end{equation*}
Then we have
\begin{align*}
1 + [(b-1) + (b-1) b + (b-1) b^2 + \cdots ]  = & 0 + (1 + (b-1)) b + (b-1) b^2 + \cdots \\  = & 0 + 0 b + (1 + (b-1)) b^2 + \cdots \\ = & 0 b + 0 b^2 + \cdots = 0.
\end{align*}
The set $\mathbb{Z}_b$ with this addition then forms an abelian group.

The set of $b$-adic numbers has various applications to QMC theory. In the following we present one example in the context of lattice point sets. Other examples are to be found, for example, in \cite{hel2009,hel2010,pill2013}.

\subsubsection*{Example: extensible lattice point sets}

One disadvantage of rank-1 lattice point sets is their dependence on the
cardinality $N$ of the resulting point set. If one constructs a generating vector of a
 lattice rule of cardinality $N$ with good quality, it does not mean that
the same vector can be used to generate a lattice point set of good quality
which uses $N' \not= N$ points. 

Extensible lattice rules have the property that the number $N$ of points in the node set may be increased while retaining the existing points. Their definition is based on $b$-adic numbers. Let $\bsa \in \ZZ_b^s$ and define the infinite sequence $\cS_{\bsa}=(\bsx_n)_{n \ge 0}$ by $\bsx_n=\{\bsa \phi_b(n)\}$, where $\phi_b$ is the $b$-adic radical inverse function as defined in Section~\ref{secDig} and where the fractional part function $\{\cdot\}$ is applied component-wise. 

The so constructed sequence has the property that any initial segment with $N=b^m$ points is a rank-1 lattice point set. Indeed, for $m \in  \NN$ and $\bsa_m:= \bsa \pmod{b^m}$ (applied component-wise) we have $$\{ \{\bsa \phi_b(n)\} \ : \ n=0,1,\ldots,b^m-1\} = \left\{ \left\{\frac{\ell}{b^m} \bsa_m\right\}\ : \ \ell =0,1,\ldots,b^m -1\right\}=\cP(\bsa_m,b^m).$$ Furthermore, for $\overline{m} \ge m$ we have $\cP(\bsa_m,b^m) \subseteq \cP(\bsa_{\overline{m}},b^{\overline{m}}).$

It has been shown by Hickernell and Niederreiter \cite{HN2003} that there exist $\bsa \in \ZZ_b^s$ such that for all $\varepsilon>0$ $$D_N^\ast(\cS_{\bsa}) \ll_{s,\varepsilon} \frac{(\log N)^{s+1} (\log \log N)^{1+\varepsilon}}{N} \ \ \mbox{ for all }\ N=b,b^2,b^3,\ldots.$$  More results on extensible lattice point sets can be found in \cite{CKN06,DPW06,HHLL,NP2009}.

\subsection{Diophantine Approximation}\label{secDioph}

Diophantine approximation deals with the problem of approximating real numbers by rational numbers, or, in the multivariate case, of approximating real vectors by rational vectors. In dimension one the theory of continued fractions plays an utmost important role in this field. But also in the multivariate case there are many important theorems in this area such as Dirichlet's approximation theorem or Minkowski's theorem on linear forms which is a corollary to Minkowski's fundamental theorem (Theorem~\ref{thmink}); see, for example, \cite{baker,cass,HW}:
\begin{theorem}[Dirichlet]\label{thDir}
Let $\alpha_1,\ldots,\alpha_s \in \RR$. Then there exists a vector $(p_1,\ldots,p_s,q) \in \ZZ^s\times \NN$, such that $$|q \alpha_j -p_j| \le q^{1/s}\ \ \mbox{ for all }\ j=1,2,\ldots,s.$$ Moreover, if at least one $\alpha_j$ is irrational, then there are infinitely many tuples $(p_1,\ldots,p_s,q) \in \ZZ^s \times \NN$ with this property.
\end{theorem}

\begin{theorem}[Minkowski]\label{thMinLin}
Let $A=(a_{i,j})_{i,j=1}^n$ be a real matrix and let $c_1,\ldots,c_n \in \RR^+$. Consider the $n$ linear forms $$L_i(x_1,\ldots,x_n)=\sum_{j=1}^n a_{i,j} x_j\ \ \ \mbox{ for }\ i=1,2,\ldots,n.$$ Then the following holds: if $c_1 \cdots c_n \ge |{\rm det}(A)|$, then there exists a vector $(h_1,\ldots,h_n)\in \ZZ^n \setminus\{\bszero\}$ such that $|L_1(h_1,\ldots,h_n)| \le c_1$ and $|L_i(h_1,\ldots,h_n)| < c_i$ for all $i=2,\ldots,s$.
\end{theorem}

The applications of Diophantine approximation to QMC, in particular to discrepancy theory, are various and numerous and  cannot all be cited here. We just mention some examples such as \cite{beck,DT,kuinie,nie72,nie73,nie78,niesiam}. Furthermore, applications of Diophantine approximation to QMC are not only restricted to the archimedean case. Many results have non-archimedean analogs, for example in the context of approximations of Laurent series over finite fields by rational functions, which can also be applied to problems in QMC. This plays a major role, e.g., in the analysis of digital nets and sequences such as polynomial lattice point sets or digital Kronecker sequences. See \cite{LN93,LP13,LP14,niesiam,nie95} for examples.

Here we present one classical application which is taken from \cite{nie72} and which deals with the discrepancy of Kronecker sequences (see Section \ref{secLPS}).

\subsubsection*{Example: Discrepancy of Kronecker sequences}

One main problem in the theory of Diophantine approximation is to find bounds for $\|\bsh \cdot \bsalpha\|$, where $\| \cdot\|$ denotes the distance to the nearest integer function, i.e., $\|x\|=\min(\{x\},1-\{x\})$ for $x \in \RR$ and where $\bsalpha=(\alpha_1,\ldots,\alpha_s)$ and $\bsh \in \ZZ^s$. This problem is directly linked to the discrepancy of Kronecker sequences $(\{n \bsalpha\})_{n \ge 0}$. It is well known (and can easily be deduced from Weyl's criterion) that a Kronecker sequence is uniformly distributed if and only if $1,\alpha_1,\ldots,\alpha_s$ are linearly independent over the rationals.

\begin{definition}\rm
For a real number $\eta$, an $s$-tuple $\bsalpha \in (\RR \setminus \QQ)^s$ is said to be of {\it type} $\eta$, if $\eta$ is the infimum of all numbers $\sigma$ for which there exists a positive constant $c=c(\sigma,\bsalpha)$ such that $$r(\bsh)^{\sigma} \| \bsh \cdot \bsalpha\| \ge c\ \ \ \mbox{ for all }\ \bsh \in \ZZ^s \setminus \{\bszero\},$$ where  $r(\bsh) =\prod_{j=1}^s \max(1,|h_j|)$ for $\bsh=(h_1,\ldots,h_s)\in \RR^s$.
\end{definition}

The following result follows easily from the above two theorems:

\begin{proposition}
The type $\eta$ of an irrational vector $\bsalpha$ is at least one.
\end{proposition}

\begin{proof}
Assume that the type $\eta$ of $\bsalpha \in (\RR\setminus \QQ)^s$ is less then one. According to Theorem~\ref{thDir} there exist infinitely many $(p_1,\ldots,p_s,q) \in \ZZ^s\times \NN$ such that $|q \alpha_i-p_i| \le q^{-1/s}$ for all $i=1,2,\ldots,s$. Now consider the linear forms $L_i(x_1,\ldots,x_s,x)=x_i$ for $i=1,2,\ldots,s$ and $L_{s+1}(x_1,\ldots,x_s,x)=p_1 x_1+\cdots+p_s x_s-q x$ with $q$ as absolute value of the corresponding determinant. According to Theorem~\ref{thMinLin} there exists a vector $(h_1,\ldots,h_s,h)\in \ZZ^{s+1}\setminus\{\bszero\}$ such that $$|h_j| \le q^{1/s} \ \ \mbox{ for all } \ j=1,2,\ldots,s \ \mbox{ and } \ |h_1p_1+\cdots+h_s p_s-qh|<1.$$ Since $h_1p_1+\cdots+h_s p_s-qh \in \ZZ$ we obtain that $qh=h_1p_1+\cdots+h_s p_s$. 

Now for any $\sigma \in (\eta,1)$ we have (recall that $q\ge 1$)
\begin{align*}
|h_1\alpha_1+\cdots +h_s \alpha_s -h| r(\bsh)^{\sigma} & \le  |h_1 q \alpha_1+\cdots +h_s q \alpha_s -q h| \frac{1}{q} \prod_{j=1}^s \max(1,q^{1/s})^{\sigma}\\
& = | h_1(q \alpha_1-p_1)+\cdots +h_s (q \alpha_s-p_s)| \frac{1}{q^{1-\sigma}} \le \frac{s}{q^{1-\sigma}}
\end{align*}
and hence $$\inf_{\bsh \not=\bszero} r(\bsh)^{\sigma} \|\bsh \cdot \bsalpha\| \le \frac{s}{q^{1-\sigma}}$$ for infinitely many $q \in \NN$. Thus the infimum is zero and the result follows.
\end{proof}

On the other hand it has been shown by Schmidt \cite{schm1970} that $\bsalpha=(\alpha_1,\ldots,\alpha_s)$, with real algebraic components for which $1,\alpha_1,\ldots,\alpha_s$ are linearly independent over $\QQ$, is of type $\eta=1$. In particular, $({\rm e}^{r_1},\ldots,{\rm e}^{r_s})$ with distinct nonzero rationals $r_1,\ldots,r_s$ or $(\sqrt{p_1},\ldots,\sqrt{p_s})$ with distinct prime numbers $p_1,\ldots,p_s$ are of type $\eta=1$.

\begin{theorem}[Niederreiter]
Let $\bsalpha$ be an $s$-tuple of irrationals of type $\eta=1$. Then the discrepancy of the Kronecker sequence $\cS_{\bsalpha}=(\{n \bsalpha\})_{n \ge 0}$ satisfies for all $\varepsilon>0$ $$D_N(\cS_{\bsalpha}) \ll_{s,\varepsilon}\frac{1}{N^{1-\varepsilon}}.$$
\end{theorem}

\begin{proof}
The proof is according to \cite{nie72}. Using the formula for a geometric sum we obtain 
\begin{align*}
\left|\sum_{n=0}^{N-1} \exp(2 \pi \icomp \bsh \cdot \bsx_n) \right| & = \left|\sum_{n=0}^{N-1} \exp(2 \pi \icomp \bsh \cdot \bsalpha)^n \right|\\
& \le  \frac{2}{|\exp(2 \pi \icomp \bsh \cdot \bsalpha)-1|}= \frac{1}{|\sin(2 \pi \bsh \cdot \bsalpha)|} \le \frac{1}{2 \| \bsh \cdot \bsalpha\|}.
\end{align*}
Inserting this into the Erd\H{o}s-Tur\'{a}n-Koksma inequality (Theorem~\ref{thETH}) we obtain for all $m \in \NN$
$$D_N(\cS_{\alpha}) \ll_s \frac{1}{m}+\frac{1}{N} \sum_{0 < |\bsh|_{\infty} \le m} \frac{1}{r(\bsh)} \frac{1}{\| \bsh \cdot \bsalpha\|}.$$ Now we use the identity
$$\sum_{0 < |\bsh|_{\infty} \le m} \frac{1}{r(\bsh)} \frac{1}{\| \bsh \cdot \bsalpha\|} = \sum_{n_1,\ldots,n_s=1}^m f(n_1,\ldots,n_s) \sum_{\bsh \in \ZZ^s\setminus\{\bszero\}\atop |h_j| \le n_j \ \forall j} \frac{1}{\| \bsh \cdot \bsalpha\|},$$ 
where $f(n_1,\ldots,n_s)=\prod_{j=1}^s g_m(n_j)$ and where $g_m(n)=1/(n(n+1))$ if $n\in \{1,\ldots,m-1\}$ and $g_m(m)=1/m$. This can be shown by computing the total coefficient of $1/\|\bsh \cdot \bsalpha\|$ on the right-hand side of the equation (see \cite[p.~222]{nie72} for details). 

In a first step we estimate the inner sum of the above double sum. Since $\bsalpha$ is of type one we obtain for all $\bsh,\bsh' \in \ZZ^s\setminus \{\bszero\}$ satisfying $|h_j|, |h_j'| \le n_j$ for all $j=1,\ldots,s$, and $\bsh \not=\pm \bsh'$ that $$\| \bsh \cdot \bsalpha \pm \bsh' \cdot \bsalpha\| =\|(\bsh \pm \bsh')\cdot\bsalpha\|\ge c r(\bsh+\bsh')^{-1-\varepsilon} \ge c r(2 \bsn)^{-1-\varepsilon}=:d$$ for all $\varepsilon >0$, where $c=c(\varepsilon,\alpha)$ and where $\bsn=(n_1,\ldots,n_s)$. Since $\|x \pm y\| \le | \|x\|-\|y\||$ we obtain $$|\ \|\bsh \cdot \bsalpha\| - \| \bsh' \cdot \bsalpha\| \ | \ge d.$$ Hence in each of the intervals $[k d, (k+1)d)$ for $k=0,1,\ldots, \lfloor 1/(2d)\rfloor$, there can lie at most two numbers of the form $\|\bsh \cdot \bsalpha\|$, with no such number in the interval $[0,d)$, since we also have $\|\bsh \cdot \bsalpha\| \ge d$. Therefore
$$\sum_{\bsh \in \ZZ^s\setminus\{\bszero\}\atop |h_j| \le n_j \ \forall j} \frac{1}{\| \bsh \cdot \bsalpha\|} \le 2 \sum_{k=1}^{\lfloor 1/(2d)\rfloor} \frac{1}{kd} \le \frac{2}{d}(1+\log \lfloor 1/(2d)\rfloor) \ll_{s,\varepsilon} r(\bsn)^{1+2 \varepsilon}.$$

Now we obtain
\begin{align*}
\sum_{0 < |\bsh|_{\infty} \le m} \frac{1}{r(\bsh)} \frac{1}{\| \bsh \cdot \bsalpha\|} & \ll_{s,\varepsilon}  \sum_{n_1,\ldots,n_s=1}^m f(n_1,\ldots,n_s) (n_1n_2\cdots n_s)^{1+2 \varepsilon}\\
& = \left(\sum_{n=1}^m  g_m(n) n^{1+2 \varepsilon}\right)^s \ll m^{2 s \varepsilon},
\end{align*}
where the last estimate easily follows from the definition of $g_m$. Finally we obtain 
$$D_N(\cS_{\alpha}) \ll_s \frac{1}{m}+\frac{m^{2 s \varepsilon}}{N}$$ and the result follows by choosing $m=N$.
\end{proof}


\section{Probability Theory}\label{sec_prob}

The probabilistic method in general is used to show the existence of mathematical objects with certain properties by considering a probability measure on a class of objects and proving that the probability that a random object has the desired properties is positive or even close to 1. 
This concept is crucially used in many existence proofs in QMC.  

\subsection{Hoeffding's Inequality}\label{HI}

Often, one wants to construct an object satisfying many constraints. 
Using the probabilistic method, the simplest way to achieve this is to show that the probability that one constraint is not satisfied is extremely small and then applying a union bound over all constraints. 
Extremely small probabilities can be obtained for the deviation from the mean for sums of independent random variables. 
A general and useful tool in the case of bounded random variables is Hoeffding's inequality \cite{hoe1963}.

\begin{theorem}[Hoeffding]\label{thm:hoein}
 Let $X_1,\dots,X_N$ be independent real valued random variables such that $a_i \le X_i - \EE(X_i) \le b_i$ for $i=1,\dots,N$ almost surely.
 Then for all $t>0$
 $$
   \mathrm{Prob} \left( \left|\sum_{i=1}^N (X_i-\EE(X_i)) \right| > t \right) \le 2 \exp \left( -\frac{2t^2}{\sum_{i=1}^N (b_i-a_i)^2}\right).
 $$
 In particular, if $\EE(X_i) =0$ and $|X_i| \le 1$ almost surely for $i=1,\dots,N$, then
 $$
   \mathrm{Prob} \left( \left|\sum_{i=1}^N X_i \right| > t \right) \le 2 \exp \left( - \frac{t^2}{2N}\right).
 $$
\end{theorem}

\subsubsection*{Example: Discrepancy of random points}

This approach was used in \cite{HNWW2001} to give an explicit bound for the star discrepancy showing polynomial tractability of the star discrepancy. For different notions of tractability and their extensive studies we refer to \cite{NW08,NW10,NW12}.

\begin{theorem}[Heinrich, Novak, Wasilkowski, Wo\'zniakowski]\label{thm:dischoeff}
 For $N,s \in \NN$, there exists an $N$-element point set ${\cal P}$ in $[0,1)^s$ satisfying the discrepancy bound
 $$ D_N(\cP) \ll \left(\frac{s}{N}\right)^{1/2} ( \log s+\log N)^{1/2} .$$
\end{theorem}

\begin{proof}[Sketch]
 Let ${\cal P}=\{\bst_1,\dots,\bst_N\}$ where $\bst_1,\dots,\bst_N$ are independent and uniformly distributed in $[0,1)^s$.
 We want to show that
 $$ \mathrm{Prob} \left( D_N(\cP) \le 2 \varepsilon \right) > 0 $$
 where $2 \varepsilon$ is the right hand side in Theorem~\ref{thm:dischoeff}.
 That amounts to the task to show that the event
 $$ D_N(\cP,\bsx) > 2 \varepsilon \ \ \mbox{at least for one } \bsx\in[0,1)^s$$
 has a probability smaller than 1. 
 These are infinitely many constraints, but it can be shown that 
 $ D_N(\cP,\bsx) > 2 \varepsilon$ implies  $ D_N(\cP,\bsy) > \varepsilon$ for one of the points   
 in a rectangular equidistant grid $\Gamma_{m,s}$ of mesh size $\frac{1}{m}$ with $m=\lceil s/\varepsilon \rceil$.
 Actually, this holds either for the grid point directly below left or up right from $\bsx$.
 Since the grid $\Gamma_{m,s}$ has cardinality $(m+1)^s$, a union bound shows that it is enough to prove
 $$  \mathrm{Prob} \left( D_N(\cP,\bsx) >  \varepsilon \right) < (m+1)^{-s} $$
 for every $\bsx \in \Gamma_{m,s}$.
 But now  
 $$ N D_N(\cP,\bsx) = \sum_{i=1}^N \left( {\mathbf 1}_{B_{\bsx}}(\bst_i) -  \, {\rm vol} (B_{\bsx}) \right) $$
 is the sum of the $N$ random variables $X_i={\mathbf 1}_{B_{\bsx}}(\bst_i) -  \, {\rm vol} (B_{\bsx})$, which have mean 0 and
 obviously satisfy $|X_i|\le 1$.
 So we can apply Hoeffding's inequality and obtain
 $$ \mathrm{Prob} \left( D_N(\cP,\bsx) >  \varepsilon \right)  = \mathrm{Prob} \left( \left|\sum_{i=1}^N X_i \right| > N\varepsilon \right) \le 2 \exp \left( \frac{-N \varepsilon^2}{2}\right) < (m+1)^{-s}, $$
 where the last inequality is satisfied for the chosen values of the parameters.
\end{proof}

It should be mentioned that this approach can be easily improved and used to construct low-discrepancy points algorithmically. For more information we refer to the survey article \cite{Gne12} and the references therein.

\subsection{Vapnik-\v{C}ervonenkis Classes and Empirical Processes}\label{VCEP}

The behavior of the discrepancy function $D_N(\cP, \,\cdot\,)$ for a point set ${\cal P}=\{\bst_1,\dots,\bst_N\}$ with independent and uniformly distributed $\bst_1,\dots,\bst_n$ as already considered in the previous section is intimately connected with the theory of empirical processes. In particular, this yields an essential improvement of Theorem~\ref{thm:dischoeff} in \cite{HNWW2001}. Very general notions of the discrepancy function are related to empirical processes. Average discrepancies are then expectations of certain norms of such empirical processes as we explain below.

Let us first explain what an empirical process is. For a fixed integer $N$, let $X_1,\dots, X_N$ be independent and identically distributed random variables defined on the same probability space with
values in some measurable space $M$. Assume that we are given a sufficiently small class ${\cal F}$ of measurable real functions on $M$. The {\em empirical process} indexed by ${\cal F}$ is given by
$$ \alpha_N(f) = \frac{1}{\sqrt{N}} \sum_{i=1}^N \big( f(X_i) - {\mathbb E}(f(X_i)) \big) \ \ \ \mbox{ for }\ f\in {\cal F}.$$

Now let $X_i=\bst_i$  and let ${\cal F}$ be the class of functions ${\mathbf 1}_{B(x)}$ with $x\in [0,1)^s$.
Then 
$$ \alpha_N \big({\mathbf 1}_{B_{\bsx}} \big) = \sqrt{N} D_N(\cP, \bsx) $$
for $\bsx\in[0,1)^s$, so $\sqrt{N} D_N(\cP, \,\cdot\,)$ is an empirical process. The expectation of the star discrepancy is related to the expectation of the supremum of this empirical process via
$$  \sqrt{N} \, \EE(D_N(\cP)) = \EE\left(\sup_{\bsx} \left| \alpha_N \big({\mathbf 1}_{B_{\bsx}} \big) \right|\right). $$

Now Donsker's Theorem \cite{DON} from empirical process theory tells us that for any fixed $s\in\NN$ we have
$$  \sqrt{N} \, \EE(D_N(\cP)) \to \EE\left(\sup_{\bst\in [0,1]^s} | \widetilde{B}_s(\bst) |\right) $$
for $N\to\infty$. Here $\widetilde{B}_s$ refers to the $s$-dimensional pinned Brownian sheet.
It seems to be open what the value on the right hand side is for $s>1$, so also the exact determination of $\EE(D_N(\cP))$ is probably difficult.

But estimates for the supremum of empirical processes are important and available for certain classes of index sets.
One example are Vapnik-\v{C}ervonenkis classes which we introduce now.
Let $(X,\mathcal{F},\mathrm{Prob})$ be a probability space. A countable family $\mathcal{C}$ of
measurable subsets of $X$ is called a {\it Vapnik-\v{C}ervonenkis class}  (for short VC-class) if there exists a nonnegative integer $s$ such
that
$$ \# \{ A \cap C \,:\, C\in \mathcal{C} \} < 2^{s+1} $$
for any subset $A\subset X$ with $|A|= s+1$. The smallest such $s$ is called VC-dimension of $ \mathcal{C}$.

Also the discrepancy function can be generalized to this setting as follows.
The discrepancy of an $N$-element set $\cP = \{\bst_1,\ldots,\bst_N\} \subseteq X$ with respect to
$C\in \mathcal{C}$ is given as
$$ D_N(\cP,C) =  \frac{1}{N} \sum_{i=1}^N {\mathbf 1}_C(\bst_i) - \mathrm{Prob}(C).$$
Furthermore, let $$ D_N(\cP) = \sup_{C\in \mathcal{C}} \big| D_N(\cP,C) \big|.$$
If we choose for $\mathcal{C}$ the class of boxes $B_{\bsx}$ with $\bsx \in  [0,1]^s$, then we obtain 
the classical notion of the star discrepancy. Moreover, this class is a VC-class of dimension $s$, see \cite{Dud1984}.
Choosing $\bst_i=X_i$ as independent random variables identically distributed according to $\mathrm{Prob}$, we can again treat $D_N(\cP)$ as the supremum of an empirical process indexed by the VC-class $\mathcal{C}$.

The following theorem is a crucial large deviation inequality for empirical processes on VC-classes due to Talagrand \cite{tal1994} and Haussler \cite{Hau1995}.

\begin{theorem}[Talagrand, Haussler]\label{thmTH}
  There is a positive number $K$ such that for all VC-classes of dimension $s$, probabilities $\mathrm{Prob}$,  $c \ge K s^{1/2}$ and $N\in \NN$ 
	$$ \mathrm{Prob} \left( D_N(\cP) \ge c N^{-1/2} \right) \le \frac{1}{c} \left( \frac{Kc^2}{s}\right)^s \exp(-2s^2).$$
\end{theorem}

Using this estimate instead of Hoeffding's inequality as in the previous section, one arrives at the following sharpening of Theorem \ref{thm:dischoeff}
also proved in \cite{HNWW2001}.

\begin{theorem}[Heinrich, Novak, Wasilkowski, Wo\'zniakowski]\label{thm:discVC}
 For $N,s \in \NN$, there exists an $N$-element point set ${\cal P}$ in $[0,1)^s$ satisfying the discrepancy bound
 $$ D_N(\cP) \ll \left(\frac{s}{N}\right)^{1/2}.$$
\end{theorem}

For a version with an explicit constant in this inequality we refer to \cite{ais2011}, for a lower bound for arbitrary sets to \cite{hin2004}, and for
a corresponding lower bound of the expectation of the star discrepancy of a random point set to \cite{doe2014}. A standard reference for empirical processes
is \cite{vdVW1996}. 

Donsker's Theorem also yields that for any fixed $s\in\NN$ we have
$$  \sqrt{N} \, \left( \EE(D_N(\cP,L_p)^p) \right)^{1/p} \to \left( \EE\left(\int_{[0,1]^s} | \widetilde{B}_s(\bst) |^p \right) \dint \bst \right)^{1/p}$$ 
for $N\to\infty$ and $0<p<\infty$. The speed of convergence of average $L_p$-discrepancies on the left hand side and quantitative estimates for fixed $N$  
are studied in \cite{HW2012,ste2010}.

\vspace{1cm}

\noindent {\bf Acknowledgments.} The authors would like 
to thank three anonymous referees for suggestions and comments.

\vspace{1cm}
\noindent{\bf Author's Addresses:}\\

\noindent Josef Dick, School of Mathematics and Statistics, The University of New South Wales, Sydney, NSW 2052, Australia.  Email: josef.dick(at)unsw.edu.au \\

\noindent Aicke Hinrichs, Institut f\"ur Mathematik, Universit\"at Rostock , Ulmenstra{\ss}e 69, D-18051 Rostock, Germany. Email: aicke.hinrichs(at)uni-rostock.de\\

\noindent Friedrich Pillichshammer, Institut f\"{u}r Finanzmathematik, Johannes Kepler Universit\"{a}t Linz, Altenbergerstra{\ss}e 69, A-4040 Linz, Austria. Email: friedrich.pillichshammer(at)jku.at

\label{lastpage}

\end{document}